\documentclass[11pt]{article}

\usepackage{amsmath, amsfonts,amssymb, amsthm}
\usepackage{mathtools}
\usepackage{multirow}
\usepackage{graphicx}
\usepackage{color}
\usepackage[table]{xcolor}
\usepackage{epsf}
\usepackage{tikz}
\usepackage{subfig}
\usepackage{hyperref}
\hypersetup{
    colorlinks = true,
    linkcolor = {black},
    citecolor = {black}
}
\usepackage{pgfplots}
\pgfplotsset{compat=1.18}

\usepackage[left=1.1in, right=1.1in, top=1.1in, bottom=1.1in]{geometry}

\renewenvironment{thebibliography}[1]{
  \begin{oldthebibliography}{#1}
    \setlength{\itemsep}{.4em}
    \setlength{\parskip}{0em}
}
{
  \end{oldthebibliography}
}

\usepackage{enumitem}
\setlist{itemsep=0pt, topsep=0pt}

\newcommand{\floor}[1]{\lfloor#1\rfloor}
\newcommand{\ceiling}[1]{\lceil#1\rceil}

\newcommand{\ex}{\mathrm{ex}}

\newcommand{\RR}{\mathbb{R}}

\newcommand{\tbf}[1]{\textbf{#1}}

\newtheorem{theorem}{Theorem}[section]

\newtheorem{lemma}[theorem]{Lemma}

\newtheorem{claim}[theorem]{Claim}

\newtheorem{observation}[theorem]{Observation}

\newtheorem{example}[theorem]{Example}
\newtheorem{problem}[theorem]{Problem}

\newtheorem{conjecture}[theorem]{Conjecture}
\newtheorem{remark}[theorem]{Remark}

\newenvironment{proofclaim}[1][Proof of claim]{\begin{proof}[#1]}{\end{proof}}

\newcommand{\ep}{\epsilon}

\usepackage{color}

\title{On the Ramsey numbers of wheels, cycles, and stars}
\author{Louis DeBiasio\thanks{Department of Mathematics, Miami University, Oxford, OH. \texttt{debiasld@miamioh.edu}. Research supported in part by NSF grant DMS-1954170 and AMS-Simons Research Enhancement Grants for PUI Faculty GR001015.}, Tucker Wimbish\thanks{Department of Mathematics, Miami University, Oxford, OH. \texttt{wimbistj@miamioh.edu}.}}
\date{\today}

\begin{document}

\maketitle

\begin{abstract}
The wheel $W_{k}$ is the graph on $k+1$ vertices consisting of a vertex joined to a cycle of length $k$, and we say that $W_k$ is an even wheel if $k$ is even.  Mao, Wang, Magnant, Schiermeyer \cite{MWMS} proved that the Ramsey number of $W_{2n}$ is between $4n+1$ and $12n-2$.  We improve both of these bounds, showing that $5n-\frac{1+(-1)^{n}}{2}\leq R(W_{2n})\leq 8n+664$ for all integers $n\geq 2$.

The main focus of the paper concerns two general results on the Ramsey numbers of stars versus even wheels and even cycles versus even wheels, from which the above bounds are obtained as a corollary. That is, we asymptotically determine $R(K_{1,m}, W_{2n})$ and $R(C_{2m}, W_{2n})$ for all sufficiently large $m$ and $n$, both of which were open problems for most regimes. 
%In particular, as it directly relates to the Ramsey number of $R(W_{2n})$, we prove that for all $n\geq 1000$, $R(K_{1,2n}, W_{2n})=5n-\frac{1+(-1)^{n}}{2}$ and for all $n\geq 2$, $R(C_{2n}, W_{2n})\leq 4n+332$.  

As for odd wheels, we note that the analogous values for stars versus odd wheels and odd cycles versus odd wheels were already known exactly, from which it follows that 
$6n+4=R(K_{1,2n+1}, W_{2n+1})\leq R(W_{2n+1})\leq 2\cdot R(C_{2n+1}, W_{2n+1})=12n+2$. Very recently, Zhang and Chen \cite{ZC} improved the upper bound to $R(W_{2n+1})\leq \frac{32n}{3}+O(1)$.  We are able to refine their proof, further improving the upper bound to $R(W_{2n+1})\leq 10n+O(1)$.
\end{abstract}

\section{Introduction}

\subsection{Ramsey}

The Ramsey number of a graph $H$, denoted $R(H)$ is the smallest $N$ such that in every $2$-coloring of the edges of $K_N$ there is a monochromatic copy of $H$. Given graphs $H_1$ and $H_2$ we write $R(H_1, H_2)$ to be the smallest $N$ such that in every 2-coloring of the edges of $K_N$ there exists a copy of $H_1$ in color 1 or a copy of $H_2$ in color 2. Note that if $H_1\subseteq G_1$ and $H_2\subseteq G_2$, then $R(H_1, H_2)\leq R(G_1, G_2)$.  If $H_1=H=H_2$, we typically write $R(H)$ instead of $R(H,H)$.  The $k$-wheel, $W_{k}$, is the graph on $k+1$ vertices obtained by joining a vertex to a cycle with $k$ vertices.  For completeness, we can define $W_1$ to be $K_2$ and $W_2$ to be $K_3$.  We warn the reader that some authors write $W_k$ to be the graph on $k$ vertices obtained by joining a vertex to a cycle of length $k-1$.  The $k$-fan, denoted $F_k$, is the graph on $2k+1$ vertices obtained by joining a vertex to a matching of size $k$.

The value of $R(F_n)$ is well-studied, and in a recent paper \cite{DW} we improved the best known upper and lower bounds on $R(F_n)$ by showing $(3+\sqrt{3})n-8<R(F_n)\leq (5+o(1))n$.  Note that $F_n$ is a spanning subgraph of $W_{2n}$ so we have $R(F_n)\leq R(W_{2n})$.  This relationship prompted us to investigate the value of $R(W_{2n})$.

It is known that $R(W_4)=15$ (see  \cite{HarM} and \cite{Hen}) and $R(W_6)=19$ (see \cite{LP} and \cite{Vo}), but the value of $R(W_{2n})$ is unknown for all $n\geq 4$.  The best known bounds are due to Mao, Wang, Magnant, and Schiermeyer \cite{MWMS} who proved that for all $n\geq 2$, $4n+1\leq R(W_{2n})\leq 12n-2$.  

We improve both of these bounds as follows.

\begin{theorem}\label{thm:main}
For all integers $n\geq 2$, $$5n-\frac{1+(-1)^{n}}{2}\leq R(W_{2n})\leq 8n+664.$$
\end{theorem}

Our proof of Theorem \ref{thm:main} begins with the simple observation that
\begin{equation}\label{eq:main}
R(K_{1,2n}, W_{2n})\leq R(W_{2n}, W_{2n})\leq 2\cdot R(C_{2n}, W_{2n}).  
\end{equation}
Indeed, the lower bound holds because $K_{1,2n}\subseteq W_{2n}$, and the upper bound holds because if we have a 2-coloring of $K_N$ with $N=2\cdot R(C_{2n}, W_{2n})$, then every vertex $v$ has at least $R(C_{2n}, W_{2n})$ neighbors of some color, say red, and a copy of $C_{2n}$ inside the red neighborhood of $v$ gives a red copy of $W_{2n}$.

Given \eqref{eq:main}, we then obtain Theorem \ref{thm:main} by determining the values of $R(K_{1,2n}, W_{2n})$ (exactly) and $R(C_{2n}, W_{2n})$ (within a constant)\footnote{This latter result was independently obtained by Zhang and Chen \cite{ZC}. See the note after Theorem \ref{thm:CycleWheel} for a more detailed discussion.} respectively.  In the process of doing so, we realized that despite the large volume of literature regarding the general values $R(K_{1,m}, W_{2n})$ and $R(C_{2m}, W_{2n})$ (see \cite[Sections 4.3, 5.5, 5.8]{Rad}), there were very few regimes where the values were even bounded, let alone determined asymptotically.  

We were able to asymptotically (exactly in some cases) determine the values of $R(K_{1,m}, W_{2n})$ and $R(C_{2m}, W_{2n})$ for all $m$ and $n$ (sufficiently large in some cases).  While we obtain Theorem \ref{thm:main} as a corollary, what started out as the auxiliary problems became the main focus of the paper.  For the rest of the introduction we discuss these auxiliary problems in detail.  

\subsubsection{Stars versus wheels}

An example of Li and Schiermeyer \cite{LiSc} shows that $R(K_{1,2n}, W_{2n})\geq 5n-\frac{1+(-1)^n}{2}$, but the consequence $R(W_{2n})\geq R(K_{1,2n}, W_{2n})\geq 5n-\frac{1+(-1)^n}{2}$ seems not to have been noted before.  If the lower bound on $R(K_{1,2n}, W_{2n})$ could be improved, this would immediately improve the lower bound on $R(W_{2n})$.  The only results in the literature which give an upper bound on $R(K_{1,m}, W_{2n})$ are for the cases when $1\leq m\leq 2n-1$ (falling just short of the value we were interested in).  We show that for all integers $n\geq 1000$, $$R(K_{1,2n}, W_{2n})= 5n-\frac{1+(-1)^n}{2},$$ which implies that the lower bound on $R(W_{2n})$ cannot be improved by considering $R(K_{1,2n}, W_{2n})$ alone.  

Elaborating on what was mentioned above, it is known \cite{Has} (see \cite[Theorem 1]{LiSc}) and easy to prove (it essentially follows from Dirac's theorem, but taking into account parity conditions which allow for a regular graph or not)
    % \footnote{For the upper bound, if there is no red $K_{1,m}$, then some vertex $v$ will have blue degree at least $2n$.  Indeed, if $m$ is odd, this is clear. If $m$ is even, then $N=2n+m-1$ is odd and $m-1$ is odd, so some vertex must have red degree at most $m-2$ and consequently blue degree at least $2n$. Furthermore, every vertex in $N_B(v)$ will have at least $n$ blue neighbors in $N_B(v)$ giving a cycle of length $2n$.  For the lower bound, an $(m-1)$-regular graph (so the complement has maximum degree at most $2n-1$) on $N=2n+m-\frac{1+(-1)^m}{2}-1$ exists (since either $N$ is even or $m-1$ is even). \ld{Think about whether to include this or not.}} 
that for all $1\leq m\leq n$, $R(K_{1,m}, W_{2n})= m+2n-\frac{1+(-1)^{m}}{2}$.  For all $3\leq n< m\leq 2n-2$, Li and Schiermeyer \cite{LiSc} proved that
\begin{equation}\label{eq:LS}
R(K_{1,m}, W_{2n})= 
    \begin{cases} 
     2m+n-1, & m \text{ and } n \text{ are even} \\
     2m+n, & \text{otherwise}
   \end{cases},
\end{equation}
and \cite{HM} proves $5n-2\leq R(K_{1,2n-1}, W_{2n})\leq 5n-1$.  Note that the value of $R(K_{1,m}, W_4)$ behaves slightly differently (see \cite{LiSc}).  In the process of determining $R(K_{1,2n}, W_{2n})$, we were able to extend the proof of Li and Schiermeyer \cite{LiSc} to show the following.

\begin{theorem}\label{thm:starWheelExact}
For all integers $n$ and $m$ with $3\leq n< m\leq 3n-1000$, 
$$R(K_{1,m}, W_{2n})= 
    \begin{cases} 
     2m+n-1, & m \text{ and } n \text{ are even} \\
     2m+n, & \text{otherwise}
   \end{cases}.$$
In particular, for all $n\geq 1000$, $R(K_{1,2n}, W_{2n})=5n-\frac{1+(-1)^n}{2}$.
\end{theorem}

Additionally, we asymptotically determine $R(K_{1,m}, W_{2n})$ for all $m> n$.

\begin{theorem}\label{thm:starWheel}
For all $\ep>0$, there exists $m_0$ such that for all integers $n$ and $m$ with $m\geq m_0$ and $m>n\geq 2$, $R(K_{1,m}, W_{2n})\leq (2+\ep)m+n$.
\end{theorem}

So by combining Theorem \ref{thm:starWheel} with the existing results, we have the following (asymptotically) complete picture for all positive integers $m$ and $n$:
\begin{equation}\label{eq:starwheel}
R(K_{1,m}, W_{2n})= 
    \begin{cases} 
     m+2n-\frac{1+(-1)^{m}}{2}, & m\leq n \\
     (2+o(1))m+n, & m>n
   \end{cases}.
\end{equation}

The main thing left to do is to obtain an exact value for $R(K_{1,m}, W_{2n})$ for all $m>3n-1000$ (then as a secondary goal, remove the restriction that $n$ is sufficiently large).  

\begin{conjecture}\label{con:starwheel}
\eqref{eq:LS} holds for all integers $m\geq 2n-1\geq 4$.
\end{conjecture}

\subsubsection{Cycles versus wheels}

Recall that by \eqref{eq:main} we have $R(W_{2n})\leq 2\cdot R(C_{2n}, W_{2n})$, but unfortunately there were no results in the literature (see note after Theorem \ref{thm:CycleWheel}) which gave an upper bound on $R(C_{2n}, W_{2n})$.
It was conjectured that $R(C_{2n}, W_{2n})= 4n-1$ in \cite{SBT} (see also \cite{Shi}, \cite{ZBC2}, \cite{RZ}).  Zhang, Broersma, and Chen \cite{ZBC2} come close to addressing this problem, but they prove $R(C_{2m}, W_{2n})=4m-1$ only in the case when $m\geq n+251$.  We employ a similar approach as in \cite{ZBC2} to prove $R(C_{2m}, W_{2n})\leq 4m+O(1)$ in the remaining cases when $n\leq m\leq 250$.  In order to have a single proof which also includes the known case when $m\geq n+251$, we prove our result in the following form.

\begin{theorem}\label{thm:CycleWheel}
Let $m$ and $n$ be integers with $m\geq n\geq 2$ and set $c:=m-n$. 
We have $$4m-1\leq R(C_{2m}, W_{2n})\leq \max\{4m+332-\floor{\frac{333}{251}c}, 4m-1\}.$$  In particular, $R(C_{2n}, W_{2n})\leq 4n+332$ and $R(C_{2m}, W_{2n})= 4m-1$ for all $m\geq n+251$.
\end{theorem}

We note that after uploading this paper on arXiv, we were informed that Zhang and Chen \cite{ZC} independently proved that $R(C_{2n}, W_{2n})\leq 4n+332$ using essentially the same proof and therefore they independently obtained the corollary $R(W_{2n})\leq 8n+664$.  

When $2\leq m<n$, the Ramsey number $R(C_{2m}, W_{2n})$ seems to have not been directly studied at all (although there are some related results regarding $R(C_{2m}, K_{1,2n})$ and $R(C_{2m}, F_{n})$ when $m<n$ which we will discuss in the next subsection).  We give an asymptotically tight bound on $R(C_{2m}, W_{2n})$ for all sufficiently large $m$ with $m< n$.  This behavior is quite different depending on whether $\frac{n}{2}\leq m< n$ or $2\leq m<\frac{n}{2}$.  When $\frac{n}{2}\leq m< n$, the bounds are fairly straightforward.

\begin{theorem}\label{thm:m>n/2}
There exists constants $C$ and $m_0$ such that for all integers $n,m\geq m_0$ with $\frac{n}{2}\leq m< n$,
$$
2n+2m-2\leq R(C_{2m}, F_n)\leq R(C_{2m}, W_{2n})\leq 2n+2m+C.
$$
\end{theorem}

When $2\leq m<\frac{n}{2}$, the bounds are surprisingly more complicated.  In this case, we asymptotically determine the value of $R(C_{2m}, W_{2n})$ for sufficiently large $m<\frac{n}{2}$.  

\begin{theorem}\label{thm:m<n/2}
For all $\gamma>0$, there exists $m_0$ such that for all integers $m, n$ with $m_0\leq m< \frac{n}{2}$ and for all integers $q\geq 3$, we have
\begin{align*}
R(C_{2m}, W_{2n}) \leq 
\begin{cases} 
     (2q+\gamma)m, & \frac{q+1}{q^2}n\leq m<\frac{1}{q-1}n \\
     (2+\frac{2}{q}+\gamma)n, & \frac{n}{q}\leq m<\frac{(q+1)n}{q^2}
   \end{cases}.
\end{align*}
Furthermore, this upper bound is asymptotically tight. 
\end{theorem}
Note that in the above, $q$ is the unique integer such that $\frac{n}{q}\leq m< \frac{n}{q-1}$; i.e.~$q=\ceiling{\frac{n}{m}}$.

So by combining Theorems \ref{thm:CycleWheel}, \ref{thm:m>n/2}, and \ref{thm:m<n/2} with the existing result \cite{ZBC2} in the case $m\geq n+251$, we have the following (asymptotically) complete picture.  For all integers $m,n\geq 2$ and $q\geq 3$,
\begin{equation}\label{eq:cyclecycle}
R(C_{2m}, W_{2n})= 
    \begin{cases} 
     4m+\Theta(1), & m\geq n \\
     2m+2n+\Theta(1), & \frac{n}{2}\leq m<n\\
     \begin{cases}
    (2q+o(1))m, & \frac{(q+1)n}{q^2}\leq m<\frac{n}{q-1} \\
     (2+\frac{2}{q}+o(1))n, & \frac{n}{q}\leq m<\frac{(q+1)n}{q^2}
     \end{cases},
     & m_0\leq m<\frac{n}{2}
   \end{cases}.
\end{equation}

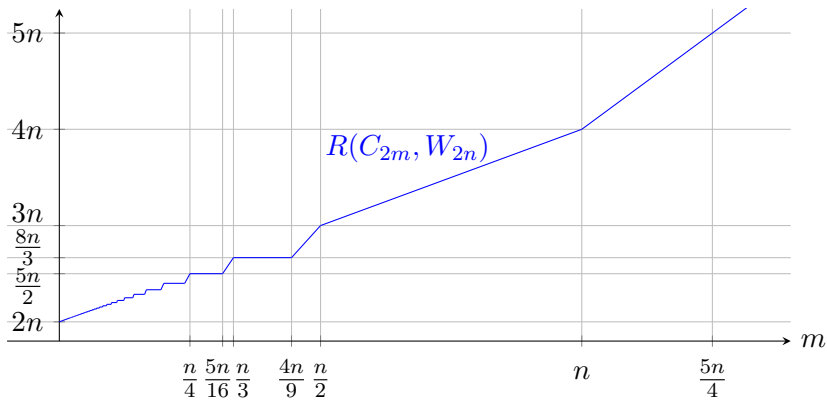
\begin{figure}[ht]
\centering
\begin{tikzpicture}
\begin{axis}[
    width=12cm, height=6cm,
    axis lines=middle,
    xmin=-0.1, xmax=1.4,
    ymin=1.8, ymax=5.25,
    x label style={anchor=west},
    xlabel={$m$},
    ylabel={},
    grid=both,
    % 1. Set the positions of the ticks
    xtick={0, 1/4, 5/16, 1/3, 4/9, 1/2, 1, 5/4},
    % 2. Set the custom labels in the exact same order
    xticklabels={
        0, 
        $\frac{\mathstrut n}{4}$, 
        $\frac{\mathstrut 5n}{16}$~~, 
        ~~$\frac{\mathstrut n}{3}$, 
        $\frac{\mathstrut 4n}{9}$, 
        $\frac{\mathstrut n}{2}$, 
        $\mathstrut n$,
        $\frac{\mathstrut 5n}{4}$, 
    }, 
    ytick={0, 1, 2, 5/2, 8/3, 3, 4, 5},
    yticklabels={0,$n$, $2n$, \raisebox{-4ex}{$\frac{5n}{2}$}, \raisebox{2.5ex}{$\frac{8n}{3}$}, \raisebox{2ex}{$3n$}, $4n$, $5n$}, 
]

% a(x): Sloped line y = 4x for x >= 1
\addplot[blue, domain=1:2.5] {4*x};

% b(x): Sloped line y = 2(1+x) for 1/2 <= x <= 1
\addplot[blue, domain=0.5:1] {2*(1+x)};

% c(x) and d(x): For 0 < x <= 1/2
% Instead of sampling q(x) = ceil(1/x), we loop through the integer steps (k = 3 to 50).
% This draws the exact geometric segments to avoid jagged edges near x=0.
\foreach \k in {3, 4, ..., 50} {
    \edef\temp{
        % d(x): The flat segments where value is 2 + 2/k
        \noexpand\draw[blue] 
            (axis cs:{1/\k}, {2 + 2/\k}) -- 
            (axis cs:{(\k+1)/(\k^2)}, {2 + 2/\k});
            
        % c(x): The sloped segments connecting the flat ones
        \noexpand\draw[blue] 
            (axis cs:{(\k+1)/(\k^2)}, {2 + 2/\k}) -- 
            (axis cs:{1/(\k-1)}, {2*\k/(\k-1)});
    }
    \temp
}

% Fill the final gap as x approaches 0 (where y converges to 2)
\draw[blue] (axis cs:0, 2) -- (axis cs:{1/50}, {2+2/50});

\node[blue, anchor=south west] at (axis cs:.49, 3.55) {$R(C_{2m}, W_{2n})$};

\end{axis}
\end{tikzpicture}
    \caption{The asymptotically exact value of $R(C_{2m}, W_{2n})$ for all sufficiently large $m,n$. For the purposes of this figure, we treat $m$ as a function of $n$.}
    \label{fig:cyclewheel}
\end{figure}

\subsubsection{Cycles versus stars, fans, and wheels}

As mentioned before, we have $K_{1,2n}\subseteq F_n\subseteq W_{2n}$ and thus 
\begin{equation}\label{eq:sfw}
R(C_{2m}, K_{1,2n})\leq R(C_{2m}, F_n)\leq R(C_{2m}, W_{2n}).
\end{equation}

Using Szemer\'edi's regularity lemma, You and Lin \cite{YL} proved $R(C_{2m}, F_n)= (4+o(1))m$ for $m\geq n$ and $R(C_{2m}, F_n)= (2+o(1))(m+n)$ for $\frac{n}{2}\leq m< n$.  Because of \eqref{eq:sfw}, our Theorems \ref{thm:CycleWheel} and \ref{thm:m>n/2} imply the results of \cite{YL} with improved bounds (in that our results hold for much smaller $m$ and $n$ and the error term is only a constant rather than linear in $n$). 

Allen et al.~\cite{ALPZ} proved that for all sufficiently large $2\leq m\leq n$ and $q\geq 2$,
%where $q\geq 2$ is the unique integer such that $\frac{2n-1}{q}< 2m-1\leq \frac{2n-1}{q-1}$,
\begin{equation}\label{eq:cyclestar}
R(C_{2m}, K_{1,2n})= 
    \begin{cases} 
      2qm-(q-1), & \frac{(q+1)(2n-1)}{q^2}< 2m-1\leq \frac{2n-1}{q-1} \\
     2n+\floor{\frac{2n-1}{q}}+1, & \frac{2n-1}{q}< 2m-1\leq \frac{(q+1)(2n-1)}{q^2}
   \end{cases}.
\end{equation}
We note that \cite{ALPZ} actually handles $R(C_{2m}, K_{1,n})$ more generally, but for comparison sake we only state their result in terms of $K_{1,2n}$ (that being said, the general result only differs by a small constant because $K_{1,2n-2}\subseteq K_{1,2n-1}\subseteq K_{1,2n}$.)

Note that by \eqref{eq:sfw}, our Theorem \ref{thm:m<n/2} asymptotically implies \eqref{eq:cyclestar} when $m\leq \frac{n}{2}$.  However, when $m\geq \frac{n}{2}$, the behavior of $R(C_{2m}, K_{1,2n})$ is different (see Figure \ref{fig:cyclewheelstar}).  The case when $\frac{n}{2}<m\leq n$ is included in \eqref{eq:cyclestar} when $q=2$.  When $n< m< 2n$, Zhang, Broersma, and Chen \cite{ZBC4} proved that $R(C_{2m}, K_{1,2n})=4n$, and when $m\geq 2n$, it follows from Dirac's theorem that $R(C_{2m}, K_{1,2n})=2m$.

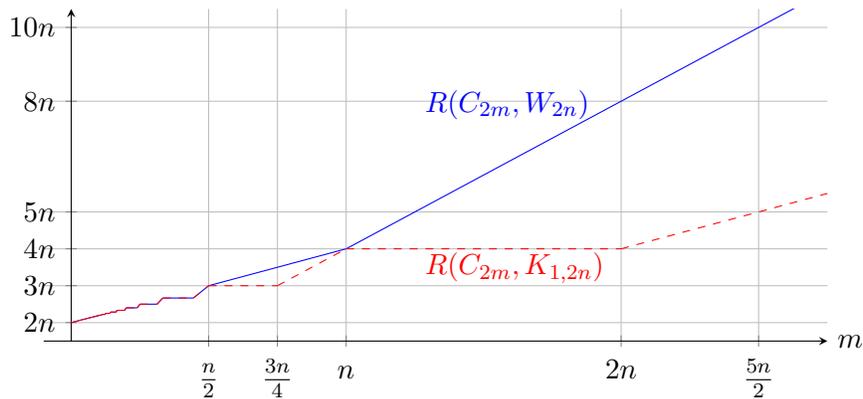
\begin{figure}[ht]
\centering
\begin{tikzpicture}
\begin{axis}[
    width=12cm, height=6cm,
    axis lines=middle,
    xmin=-0.1, xmax=2.75,
    ymin=1.5, ymax=10.5,
    x label style={anchor=west},
    xlabel={$m$},
    ylabel={},
    grid=both,
    % 1. Set the positions of the ticks
    xtick={0, 1/2, 3/4, 1, 2, 2.5},
    % 2. Set the custom labels in the exact same order
    xticklabels={0, $\frac{\mathstrut n}{2}$, $\frac{\mathstrut 3n}{4}$, $\mathstrut n$, $\mathstrut 2n$, $\frac{\mathstrut 5n}{2}$}, 
    ytick={0, 1, 2, 3, 4, 5, 8, 10},
    yticklabels={0, $n$, $2n$, $3n$, $4n$, $5n$, $8n$, $10n$}, 
]

% a(x): Sloped line y = 4x for x >= 1
\addplot[blue, domain=1:2.75] {4*x};

% b(x): Sloped line y = 2(1+x) for 1/2 <= x <= 1
\addplot[blue, domain=0.5:1] {2*(1+x)};

% a(x): Horizontal line at y = 4 for x >= 1
\addplot[red, dashed, domain=1:2] {4};

% b(x): Sloped line y = 2x for x >= 2
\addplot[red, dashed, domain=2:2.75] {2*x};

% c(x) and d(x): For 0 < x <= 1/2
% Instead of sampling q(x) = ceil(1/x), we loop through the integer steps (k = 3 to 50).
% This draws the exact geometric segments to avoid jagged edges near x=0.
\foreach \k in {3, 4, ..., 50} {
    \edef\temp{
        % d(x): The flat segments where value is 2 + 2/k
        \noexpand\draw[blue] 
            (axis cs:{1/\k}, {2 + 2/\k}) -- 
            (axis cs:{(\k+1)/(\k^2)}, {2 + 2/\k});
            
        % c(x): The sloped segments connecting the flat ones
        \noexpand\draw[blue] 
            (axis cs:{(\k+1)/(\k^2)}, {2 + 2/\k}) -- 
            (axis cs:{1/(\k-1)}, {2*\k/(\k-1)});
    }
    \temp
}

% c(x) and d(x): For 0 < x <= 1/2
% Instead of sampling q(x) = ceil(1/x), we loop through the integer steps (k = 3 to 50).
% This draws the exact geometric segments to avoid jagged edges near x=0.
\foreach \k in {2, 3, ..., 50} {
    \edef\temp{
        % d(x): The flat segments where value is 2 + 2/k
        \noexpand\draw[red, dashed] 
            (axis cs:{1/\k}, {2 + 2/\k}) -- 
            (axis cs:{(\k+1)/(\k^2)}, {2 + 2/\k});
            
        % c(x): The sloped segments connecting the flat ones
        \noexpand\draw[red, dashed] 
            (axis cs:{(\k+1)/(\k^2)}, {2 + 2/\k}) -- 
            (axis cs:{1/(\k-1)}, {2*\k/(\k-1)});
    }
    \temp
}

% Fill the final gap as x approaches 0 (where y converges to 2)
\draw[blue] (axis cs:0, 2) -- (axis cs:{1/50}, {2+2/50});

\draw[red, dashed] (axis cs:0, 2) -- (axis cs:{1/50}, {2+2/50});

\node[blue, anchor=south west] at (axis cs:1.25, 7.25) {$R(C_{2m}, W_{2n})$};

\node[red, anchor=south west] at (axis cs:1.25, 2.8) {$R(C_{2m}, K_{1,2n})$};

\end{axis}
\end{tikzpicture}
    \caption{The asymptotically exact value of $R(C_{2m}, W_{2n})$ compared to the exact value of $R(C_{2m}, K_{1,2n})$ for all sufficiently large $m,n$.}
    \label{fig:cyclewheelstar}
\end{figure}

So for all sufficiently large $m,n$ we have that $R(C_{2m}, F_n)$ and $R(C_{2m}, W_{2n})$ are asymptotically equal, and for all sufficiently large $2\leq m\leq \frac{n}{2}$ we have that $R(C_{2m}, K_{1,2n})$ and $R(C_{2m}, W_{2n})$ are asymptotically equal.  This latter result is quite surprising given that $W_{2n}$ is a far more complicated structure than $K_{1,2n}$.

\subsection{Dirac}

As a standalone result, we think it is more natural to state Theorem \ref{thm:starWheelExact} and Theorem \ref{thm:starWheel} as results about graphs with sufficiently large minimum degree.  The following well-known observation gives the link between the two formulations.

\begin{observation}\label{obs:dirac}
Let $m$ be a positive integer and let $H$ be a graph.
$R(K_{1,m}, H)\leq N$ if and only if every graph $G$ on $N$ vertices with $\delta(G)\geq N-m$ contains a copy of $H$.
\end{observation}

By Observation \ref{obs:dirac}, we have the following equivalent formulation of \eqref{eq:starwheel}.

\begin{theorem}\label{thm:dirac_wheel}
For all $\ep>0$ there exists $n_0$ such that for all integers $k, n$ with $n\geq n_0$ and $2\leq k<\frac{n}{2}$, if $G$ is a graph on $n$ vertices with \[
\delta(G)\geq 
    \begin{cases} 
     \frac{n+k}{2}+\ep n, & 2\leq k<  \frac{n}{3} \\
     2k, & \frac{n}{3}\leq k < \frac{n}{2}
   \end{cases},
\] then $W_{2k}\subseteq G$.  Furthermore, this degree condition is asymptotically tight.
\end{theorem}

We also have the following equivalent formulation of Conjecture \ref{con:starwheel} (note that Theorem \ref{thm:starWheelExact} corresponds to the case when $\frac{n+2000}{7}\leq k< \frac{n}{3}$).

\begin{conjecture}
For all integers $n$ and $k$ with $3\leq k< \frac{n}{3}$, if $G$ is a graph on $n$ vertices with  
\begin{equation}\label{eq:delta}
\delta(G)\geq  
\begin{cases} 
    \floor{\frac{n+k}{2}}, & \text{$k-1$ is odd and $\floor{\frac{n+k}{2}}$ is odd}\\
    \ceiling{\frac{n+k}{2}}, & \text{ otherwise}
   \end{cases},
\end{equation}
then $W_{2k}\subseteq G$.  Furthermore, this is best possible.
\end{conjecture}

\subsection{Tur\'an}

The \emph{Tur\'an number} of a graph $H$, denoted $\ex(n,H)$, is the maximum number of edges in an $H$-free $n$-vertex graph. As is well known, the Tur\'an number of $H$ is related to its Ramsey number in the sense that if we have a 2-coloring of $K_N$ in which one of the color classes has more than $\ex(N,H)$ edges, we would get a monochromatic copy of $H$.  This is potentially useful if $\frac{\ex(N,H)}{\binom{N}{2}}$ is not much bigger than $\frac12$.

Yuan \cite{Yuan} proved that for all $k$ there exists $n_0$ such that for all $n\geq n_0$, 
\begin{equation}\label{eq:Yuan}
\ex(n,W_{2k})=\frac{n^2}{4}+\frac{k}{4}n+\frac{k^2}{16}+o(n^2).
\end{equation}
It would be interesting to extend \eqref{eq:Yuan} to the case where $k$ is linear in $n$. 

\subsection{Odd wheels}
Analogous to \eqref{eq:main} we have 
\begin{equation}\label{eq:mainodd}
R(K_{1,2n+1}, W_{2n+1})\leq R(W_{2n+1}, W_{2n+1})\leq 2\cdot R(C_{2n+1}, W_{2n+1}),
\end{equation}
but unlike the case for even wheels, the values of $R(K_{1,2n+1}, W_{2n+1})$ and $R(C_{2n+1}, W_{2n+1})$ were already known exactly for all positive integers $n$.  

Regarding the lower bound of stars versus odd wheels, it is easy to see (see \cite{LiSc}) that
\[
R(K_{1,m}, W_{2n+1})= 
    \begin{cases} 
    m+2n+1, & m\leq n \\
    3m+1, & m> n,
    \end{cases}
\]
%\ld{the lower bound in the first case comes from a $2n$ regular graph on $m+2n$ vertices.  The lower bound in the second case comes from a complete tripartite graph with $m$ vertices in each part.}
or phrased in the minimum degree language, if $G$ is a graph on $n$ vertices with 
\[
\delta(G) \geq 
   \begin{cases} 
     \frac{2n}{3}+1 & 1\leq k\leq  \frac{n}{3} \\
     2k+1 & \frac{n}{3} \leq k \leq \frac{n-3}{2},\\
   \end{cases}
\]
then $W_{2k+1}\subseteq G$.

Regarding the upper bound of odd cycles versus odd wheels, it is known (the case $m\geq n$ is due to \cite{CCNZ} and the case $m<n$ is due to \cite{ZZC}) that for all positive integers $m,n$ with $(m,n)\neq (1,1)$, 
\[
R(C_{2m+1}, W_{2n+1})= 
\begin{cases} 
  6m+1, & m > \frac{2n}{3}\\
  4n+3, & m\leq \frac{2n}{3}.
\end{cases}
\]
In particular, $R(C_{2n+1}, W_{2n+1})=6n+1$.  

So combining these known results with \eqref{eq:mainodd}, we have $$6n+4=R(K_{1,2n+1}, W_{2n+1})\leq R(W_{2n+1}, W_{2n+1})\leq 2\cdot R(C_{2n+1}, W_{2n+1})=12n+2.$$
However, very recently, Zhang and Chen \cite{ZC} showed $R(W_{2n+1})\leq \frac{32n}{3}+O(1)$.  The most interesting aspect of their result, aside from the beautiful proof, is that it beats the trivial upper bound of $2\cdot R(C_{2n+1}, W_{2n+1})$. 

We are able to refine Zhang and Chen's proof to obtain the following further improvement.

\begin{theorem}\label{thm:oddW}
There exists $c>0$ such that for all positive integers $n$, $R(W_{2n+1})\leq 10n+c$.
\end{theorem}

So as is stands now, the best bounds for odd wheels are 
\begin{equation}\label{eq:oddbounds}
6n+4\leq R(W_{2n+1})\leq 10n+O(1).
\end{equation}

Finally, for odd wheels, the Tur\'an result follows from a result of Simonovits \cite{Sim} on edge-critical graphs (see \cite{Dz13}).  In particular, for all $3\leq k\leq \frac{n+4}{6}$,
\[
\ex(n,W_{2k+1})=\floor{\frac{n^2}{3}}.
\]
It would certainly still be interesting to determine the value of $\ex(n, W_{2k+1})$ in the case when $\frac{n+10}{6}<k\leq \frac{n}{2}-1$.

\subsection{Notation}

Let $G$ be a graph and let $X,Y\subseteq V(G)$ with $X\cap Y=\emptyset$.  We write $G[X]$ for the graph on $X$ induced by edges with both endpoints in $X$ and $G[X,Y]$ for the bipartite graph on $X\cup Y$ induced by edges with one endpoint in $X$ and the other in $Y$.  Given a graph $G$ and a 2-coloring of the edges of $G$, we write $G_i$ for the graph on $V(G)$ induced by the edges of color $i$.  We write $N_i(v)$ for the neighborhood of $v$ in $G_i$ and $d_i(v)=|N_i(v)|$ for the degree of $v$ in $G_i$.  
Often we use $B$ and $R$ for the subscripts (to mean ``blue'' and ``red'' respectively) in place of 1 and 2.

The constants in the hierarchies used to state our results are chosen from right to left.
For example, if we claim that a result holds whenever $0< a\ll b\ll c\le 1$, then 
there are non-decreasing functions $f:(0,1]\to (0,1]$ and $g:(0,1]\to (0,1]$ such that the result holds for all $0<a,b,c\le 1$  with $b\le f(c)$ and $a\le g(b)$.  
Note that $a \ll b$ implies that we may assume in the proof that, e.g., $a < b$ or $a < b^2$.

\subsection{Overview}

Section \ref{sec:starWheel} is about the Ramsey number of $K_{1,m}$ versus $W_{2n}$.
In Section \ref{sec:2low}, we give the lower bounds in Theorem \ref{thm:starWheelExact} and Theorem \ref{thm:starWheel}, and in Section \ref{sec:2up} we prove the upper bounds.  Section \ref{sec:cycleWheel} is about the Ramsey number of $C_{2m}$ versus $W_{2n}$.  In Section \ref{sec:3low}, we give the lower bounds in Theorems \ref{thm:CycleWheel}, \ref{thm:m>n/2}, and \ref{thm:m<n/2}, and in Section \ref{sec:3up} we prove the upper bounds.  In Section \ref{sec:odd} we will prove Theorem \ref{thm:oddW}. In Section \ref{sec:Conc} we discuss some further open problems.

As for the methods used in the proofs of the upper bounds, the heavy lifting in Theorems \ref{thm:starWheelExact}, \ref{thm:starWheel}, \ref{thm:CycleWheel}, \ref{thm:m>n/2}, and \ref{thm:oddW} is done by using some deep results on (weak/even) pancyclicity by Gould, Haxell, and Scott (Theorem \ref{thm:GHS}) and Brandt, Faudree, and Goddard (Theorem \ref{thm:BFG}).  These results allow one to get even cycles of a specified length in 2-connected graphs with linear minimum degree below the pancyclicity threshold.  The most difficult result in the paper is Theorem \ref{thm:m<n/2} which uses Szemer\'edi's regularity lemma together with results about connected fractional matchings.  These results work together to give us useful structural information in the case when we don't have a red $C_{2m}$, from which we can prove the existence of a blue $W_{2n}$ using some ad-hoc arguments.  

\section{Stars versus wheels}\label{sec:starWheel}

In this section we prove $R(K_{1,m}, W_{2n})=(2+o(1))m+n$ for all $m> n$ (as well as the exact value for all $3\leq n<m\leq 3n-1000$).  

\subsection{Lower bound}\label{sec:2low}

Li and Schiermeyer \cite{LiSc} give the following lower bound on $R(K_{1,m}, W_{2n})$ when $m\geq n\geq 2$.  We reproduce their example here (using our notation) for convenience.  First we state a lemma needed for the construction.

\begin{lemma}[Lemma 1 in \cite{LiSc}]\label{lem:reg}
For all integers $n$ and $k$ such that $0\leq k-1\leq n$ and $n$ or $k-1$ is even, there exists a $(k-1)$-regular graph on $n$ vertices in which every component has at most $2k-1$ vertices.  
\end{lemma}

\begin{example}
For all positive integers $m$ and $n$ with $m\geq n\geq 2$, we have 
\[
R(K_{1,m}, W_{2n})\geq 
\begin{cases} 
     2m+n-1, & m \text{ and } n \text{ are even} \\
     2m+n, & \text{otherwise}
   \end{cases}.
\]
In particular, $R(K_{1,2n}, W_{2n})\geq 5n-\frac{1+(-1)^n}{2}$.
\end{example}

\begin{proof}
If $n-1$ is even or $n+m-1$ is even (equivalently $n$ is odd or $m$ is odd), let $X, Y$ be disjoint sets where $|Y|=m$ and $|X|=n+m-1$.  Color all edges inside of $Y$ red and all edges between $X$ and $Y$ blue.  Inside of $X$, use Lemma \ref{lem:reg} so that the blue graph induced by $X$ is $n-1$ regular (and thus the red graph is $m-1$ regular) and every blue component has at most $2n-1$ vertices.  

If $n-1$ is odd and $m$ is even (equivalently $n$ and $m$ are even), let $X, Y$ be disjoint sets where $|Y|=m$ and $|X|=n+m-2$. Color all edges inside of $Y$ red and all edges between $X$ and $Y$ blue.  Inside of $X$, use Lemma \ref{lem:reg} so that the blue graph induced by $X$ is $n-1$ regular (and thus the red graph is $m-2$ regular) and every blue component has at most $2n-1$ vertices.  
\end{proof}

We also state their example using the minimum degree language.  

\begin{example}
Let $n$ and $k$ be integers with $2\leq k<\frac{n}{3}$.  There exists a graph $G$ on $n$ vertices with 
\[
\delta(G)=
\begin{cases} 
     \ceiling{\frac{n+k}{2}}-1, & \text{$k-1$ is even or $\floor{\frac{n+k}{2}}$ is even}\\
     \floor{\frac{n+k}{2}}-1, & \text{otherwise}
   \end{cases},
\]
such that $W_{2k}\not\subseteq G$.
\end{example}

\begin{proof}
First suppose $k-1$ is even or $\floor{\frac{n+k}{2}}$ is even.  In this case, let $\{X, Y\}$ be a partition of $V(G)$ with $|Y|=\ceiling{\frac{n-k}{2}}$ and $|X|=\floor{\frac{n+k}{2}}$.  Add all edges between $X$ and $Y$ and use Lemma \ref{lem:reg} inside $X$ so that $G[X]$ is a $(k-1)$-regular graph in which every component has at most $2k-1$ vertices (which is possible since $k-1$ is even or $\floor{\frac{n+k}{2}}$ is even).  The vertices in $Y$ have degree $\floor{\frac{n+k}{2}}$ and the vertices in $X$ have degree $\ceiling{\frac{n+k}{2}}-1$ and there is no copy of $W_{2k}$.  

Now suppose $k-1$ is odd and $\floor{\frac{n+k}{2}}$ is odd.  In this case, let $\{X, Y\}$ be a partition of $V(G)$ with $|Y|=\ceiling{\frac{n-k}{2}}+1$ and $|X|=\floor{\frac{n+k}{2}}-1$.  Add all edges between $X$ and $Y$ and use Lemma \ref{lem:reg} inside $X$ so that $G[X]$ is a $(k-1)$-regular graph in which every component has at most $2k-1$ vertices (which is possible since $|X|=\floor{\frac{n+k}{2}}-1$ is even).  The vertices in $Y$ have degree $\floor{\frac{n+k}{2}}-1$ and the vertices in $X$ have degree $\ceiling{\frac{n+k}{2}}$ and there is no copy of $W_{2k}$.  
\end{proof}

\subsection{Upper bound}\label{sec:2up}

In order to prove Theorems \ref{thm:starWheelExact} and \ref{thm:starWheel} we need a number of other related results. 

\begin{theorem}[Dirac \cite{Dir}]\label{thm:dir}
Let $n$ and $k$ be integers with $2\leq k\leq \frac{n}{2}$.  If $G$ is a 2-connected graph on $n$ vertices with $\delta(G)\geq k$, then $G$ has a cycle of length at least $2k$.
\end{theorem}

The following is a slight, but important, extension of the well-known Theorem \ref{thm:dir}.

\begin{theorem}[Voss, Zuluaga \cite{VZ}]\label{thm:VZ}
Let $n$ and $k$ be integers with $2\leq k\leq \frac{n}{2}$.  If $G$ is a 2-connected graph on $n$ vertices with $\delta(G)\geq k$, then $G$ contains an even cycle of length at least $2k$.  
\end{theorem}

Given a graph $G$, let $\mathrm{ec}(G)$ be the length of the longest even cycle in $G$.

\begin{theorem}[Gould, Haxell, Scott \cite{GHS}]\label{thm:GHS}
Let $d>0$ and let $C:=75\cdot 10^4/d^5$.  If $G$ is a graph on $n\geq 45C/d^4$ vertices with $\delta(G)\geq dn$, then $G$ contains a cycle of length $2t$ for all integers $4\leq 2t\leq \mathrm{ec}(G)-C$.
\end{theorem}

We also need the following lemma\footnote{While we won't reproduce the formal proof here, we note that it follows by simply deleting cut vertices until none remain.  If we do this $k-1$ times, we will have a graph on $n-(k-1)$ vertices with at least at least $k$ components, each of order greater than $\delta(G)-(k-1)>\frac{n}{k}$, a contradiction.  So this process must terminate in at most $k-2$ steps.} which was proved in \cite{ALPZ}.  

\begin{lemma}\label{lem:2con}
For all integers $n\geq k\geq 2$, if $G$ is a graph on $n$ vertices with $\delta(G)\geq \frac{n}{k}+k$, then there exists a set $X\subseteq V(G)$ with $|X|\leq k-2$ such that every component of $G-X$ is 2-connected.  
\end{lemma}

\subsubsection{Asymptotic upper bound}

We now prove Theorem \ref{thm:starWheel} which says that for all $\ep>0$ there exists $m_0$ such that for all positive integers $m$ and $n$ with $m\geq n$ and $m\geq m_0$, 
$R(K_{1,m}, W_{2n})\leq (2+\ep)m+n.$

\begin{proof}[Proof of Theorem \ref{thm:starWheel}]
Let $\ep>0$, let $C=75\cdot 10^4/\ep^5$, and let $m_0=45C/\ep^4$ (where $C$ and $m_0$ are the constants from Theorem \ref{thm:GHS}).  Let $m$ and $n$ be positive integers with $m\geq m_0$ and $m\geq n$. Let $N=(2+\ep)m+n$ and consider an arbitrary 2-coloring of $K_N$.  Suppose there is no red copy of $K_{1,m}$ and let $G$ be the graph consisting of the blue edges only.  For all $v\in V(G)$, set $G_v:=G[N(v)]$.  Note that $\delta(G)\geq N-1-(m-1)=N-m=(1+\ep)m+n$ and 
\begin{equation}\label{eq:delGv}
\delta(G_v)\geq 2\delta(G)-N=N-2m=n+\ep m.  
\end{equation}

\begin{claim}
There exists $x\in V(G)$ such that $G[N(x)]$ contains a 2-connected subgraph on at least $2n+C$ vertices with minimum degree at least $n+\frac{C}{2}$.
\end{claim}

\begin{proofclaim}
Let $v\in V(G)$ and set $G_v:=G[N(v)]$.  We will show that if $G_v$ does not satisfy the conditions of the claim, then there exists $x\in V(G_v)$ which does.  Since $\delta(G_v)\geq n+\ep m$, we can apply Lemma \ref{lem:2con} to get a set $A\subseteq N(v)$ with $|A|<\frac{3}{\ep}$ such that $G_v':=G_v-A$ only contains 2-connected components.  Furthermore, $\delta(G_v')\geq n+\ep m-|A|\geq n+\ep m/2\geq n+\frac{C}{2}$.  If $G_v'$ contains a component of order at least $2n+C$, then we are done.  So suppose that every component of $G_v'$ has order less than $2n+C$.  Let $X$ be a component of $G_v'$, let $Y=V(G_v')\setminus X$, and let $Z=V(G)\setminus (X\cup Y)$ (note that $A\subseteq Z$).  Let $x\in X$, let $X':=N(x)\cap X$, and let $Z':=N(x)\cap Z$. Note that $N(x)\cap Y=\emptyset$.  Furthermore, note that $|X'|\geq 2\delta(G)-N-|A|=n+\ep m-|A|\geq n+\ep m/2$ and that $$|Y|\geq \delta(G)-(2n+C)-|A|\geq (1+\ep)m-n-C-|A|\geq m-n+\ep m/2.$$ 

Now we claim that $G_x:=G[N(x)]$ is 2-connected.  Indeed, 
% first note that for all $u\in X'$, 
% $$d_{G_x}(u)\geq d(u)+d(x)-(N-|Y|)\geq n+\ep m+m-n+\ep m/2\geq m+\ep m$$ and for all $w\in Z'$, $d_{G_x}(w)\geq d(w)+d(x)-N\geq n+\ep m$. So we have that 
for all $u_1, u_2\in X'$ and $w\in Z'$ we have
\begin{align}
|N_{G_x}(u_1)\cap N_{G_x}(u_2)|&\geq d_{G_x}(u_1)+d_{G_x}(u_2)-d(x) \notag\\
&\geq d(u_1)+d(x)-(N-|Y|)+d(u_2)+d(x)-(N-|Y|)-d(x)\notag\\
&\geq 3\delta(G)-2N+2|Y|\notag \\
&\geq N-3m+2m-2n+\ep m\geq m-n+2\ep m\geq 2\ep m\label{eq:uu}
\end{align}
and 
\begin{align}
|N_{G_x}(u_1)\cap N_{G_x}(w)|&\geq d_{G_x}(u_1)+d_{G_x}(w)-d(x) \notag\\
&\geq d(u_1)+d(x)-(N-|Y|)+d(w)+d(x)-N-d(x)\notag\\
&\geq 3\delta(G)-2N+|Y|\geq N-3m+m-n+\ep m/2=3\ep m/2. \label{eq:uw}
\end{align}
Thus every pair of vertices in $V(G_x)$ with at least one vertex in $X'$ has at least $3\ep m/2\geq 2$ common neighbors in $G_x$, and for every pair of vertices $w_1, w_2\in Z'$, we can choose an arbitrary pair $u_1, u_2\in X'$ such that there exist distinct $y_1, y_1', y_2, y_2'$ with $y_i, y_i'\in N_{G_x}(w_i)\cap N_{G_x}(u_i)$ for all $i\in [2]$. This gives two internally disjoint paths from $w_1$ to $w_2$. Thus $G_x$ is 2-connected.
\end{proofclaim}

Let $x$ be a vertex as in the claim.  By Theorem \ref{thm:VZ}, we have that $G[N(x)]$ contains an even cycle of length at least $2n+C$.  Now by Theorem \ref{thm:GHS} we have a cycle of length $2n$ in $G[N(x)]$ which gives a copy of $W_{2n}$ in $G$.  
\end{proof}

\begin{remark}
When $n\leq m=O(n)$, the proof of Theorem \ref{thm:starWheel} can be modified to give a constant error term.  That is, we can prove $R(K_{1,m}, W_{2n})\leq 2m+n+O(1)$ by noting that while \eqref{eq:delGv} only gives $\delta(G_v)\geq n+O(1)$, we are assuming that $m$ is linear in $n$ and thus we can still apply Lemma \ref{lem:2con}.  
\end{remark}

\subsubsection{Precise upper bound}

To prove the precise upper bound we need two more auxiliary results.  Say that a graph $G$ is \emph{weakly pancyclic} if the shortest cycle in $G$ has length $a$ (i.e.~the girth of $G$ is $a$) and the longest cycle has length $b$ (i.e.~the circumference of $G$ is $c$), then $G$ contains a cycle of length $b$ for all $a\leq b\leq c$.

\begin{theorem}[Brandt, Faudree, Goddard \cite{BFG}]\label{thm:BFG}
Let $G$ be a graph on $n\geq 3$ vertices.  If $G$ is 2-connected, non-bipartite, and $\delta(G)\geq \frac{n}{4}+250$, then $G$ is weakly pancyclic (with the possible exception of the 5-cycle).
\end{theorem}

\begin{theorem}[Jackson \cite{J81}]\label{thm:jackson}
Let $k$ be an integer with $k\geq 2$.  Let $G$ be a bipartite graph with bipartition $\{X,Y\}$ such that every vertex in $Y$ has at least $k$ neighbors in $X$.  If $Y'\subseteq Y$ with $2\leq |Y'|\leq k$ and $k\leq |X|\leq 2k-2$, then $G$ has a cycle $C$ such that $V(C)\cap Y=Y'$.  In particular, $G$ has a cycle of length $2m$ for all $4\leq 2m\leq 2k$.  
\end{theorem}

We now show that $R(K_{1,m}, W_{2n})= 2m+n-\theta$ where $\theta=1$ if $m$ and $n$ are both even and $\theta=0$ otherwise.

\begin{proof}[Proof of Theorem \ref{thm:starWheelExact}]
Let $m$ and $n$ be integers with $n+1\leq m\leq 3n-1000$.  Set $N=2m+n-\theta$ where $\theta=1$ if $m$ and $n$ are both even and $\theta=0$ otherwise.  Consider a 2-coloring of $K_{N}$.  We may suppose that there is no red copy of $K_{1,m}$; i.e. every vertex has red degree at most $m-1$. Furthermore, we may assume that the red graph is maximal with respect to this property; that is, for every pair of distinct vertices $u,v$, either $u$ has red degree $m-1$ or $v$ has red degree $m-1$. Let $G$ be the graph consisting only of blue edges.  For all $u\in V(G)$, let $G_u:=G[N(u)]$.  Note that
\begin{equation}\label{eq:del'}
\delta(G)\geq (N-1)-(m-1)=m+n-\theta
\end{equation}
and 
\begin{equation}\label{eq:G_u}
\delta(G_u)\geq d(u)-m\geq n-\theta.
\end{equation}
The maximality of the red graph implies that 
\begin{equation}\label{eq:minim}
\text{for all distinct $u,v\in V(G)$, $d(u)=m+n-\theta$ or $d(v)=m+n-\theta$}.
\end{equation}

Let $v\in V(G)$ such that $d(v)=\Delta(G)$.  Note that if $\theta=1$, then since $N=2m+n-1$ is odd and $m+n-1$ is odd, we must have 
\begin{equation}\label{eq:Del}
\Delta(G)\geq m+n
\end{equation}
and thus 
\begin{equation}\label{eq:G_v'}
\delta(G_v)\geq \Delta(G)-m\geq n.
\end{equation}
Furthermore we have 
\begin{equation}\label{eq:+250'}
\delta(G_v)\geq d(v)-m\geq \frac{1}{4}d(v)+250,
\end{equation}
where the last inequality holds since $d(v)=\Delta(G)\geq m+n$ and $m\leq 3n-1000$.

First suppose that $G_v$ is bipartite.  Let $\{X, Y\}$ be the bipartition of $N(v)$ with $|X|\leq |Y|$.  We have $\delta(G_v)\leq |X|\leq \frac{d(v)}{2}$.  Let $Y'\subseteq Y$ with $|Y'|=\delta(G_v)$.  Note that every vertex in $Y'$ has at least $\delta(G_v)$ neighbors in $X$ and $$\delta(G_v)\leq |X|\leq \frac{d(v)}{2}\leq 2(d(v)-m-1)\stackrel{\eqref{eq:G_v'}}{\leq} 2\delta(G_v)-2$$ (where the second to last inequality holds since $m\leq 3n-4$). Since $2n\leq 2|Y|=2\delta(G_v)$, we may apply Theorem \ref{thm:jackson} to get a cycle of length $2n$ in $G_v[X,Y']\subseteq G_v$ which gives us a copy of $W_{2n}$ with $v$ as the center.  So we may suppose that $G_v$ is not bipartite.

Next suppose that $G_v$ is 2-connected.  By \eqref{eq:G_v'} and \eqref{eq:+250'} and the fact that $G_v$ is not bipartite, we may apply Theorem \ref{thm:dir} and Theorem \ref{thm:BFG} to get a cycle of length $2n$ in $G_v$ which gives us a copy of $W_{2n}$ with $v$ as the center.  So we may suppose that $G_v$ is not 2-connected.

We either have that $G_v$ is not connected in which case we let $X$ be the vertex set of a smallest component of $G_v$ and we set $W=\emptyset$, or $G_v$ is separable in which case we let $X$ be the vertex set of a smallest leaf-block of $G_v$ and we let $W=\{w\}$ where $w$ is the cut vertex of $G_v$ belonging to $X$.  Let $Y=V(G_v)\setminus X$ and let $Z=V(G)\setminus (X\cup Y)$.  Note that \begin{equation}\label{eq:|Y|'}
|Y|\geq \frac{\Delta(G)-|W|}{2}\geq \frac{\Delta(G)-1}{2}\geq \frac{m+n-1}{2}.  
\end{equation}

Let $x\in X\setminus W$ and if possible choose $x$ so that $|N(x)\cap (X\setminus W)|\geq n$.  Set $X'=N(x)\cap (X\setminus W)$.   We will show that either $G_x$ contains a cycle of length $2n$, or there exists $y\in Y$ such that $G_y$ contains a cycle of length $2n$.  Note that by \eqref{eq:minim}, we have 
\begin{equation}\label{eq:umin}
\text{for all $u\in N(v)=V(G_v)$, $d(u)=m+n-\theta$.}
\end{equation}

\tbf{Case 1} ($|X'|\geq n$).  Note that $N(x)\cap Y=\emptyset$ and for all $u\in X'$, $N(u)\cap Y=\emptyset$.  So for all $u\in X'$, we have 
\begin{equation}\label{eq:ux'}
d_{G_x}(u)=|N(u)\cap N(x)|\geq d(u)+d(x)-(N-|Y|)\stackrel{\eqref{eq:umin}}{=} n-\theta+|Y|\stackrel{\eqref{eq:|Y|'}}{\geq} n-\theta+\frac{\Delta(G)-1}{2},
\end{equation}
and thus for all $u_1, u_2\in X'$ we have 
\begin{align}
|N_{G_x}(u_1)\cap N_{G_x}(u_2)|&=|N(u_1)\cap N(u_2)\cap N(x)|\notag\\
&\stackrel{\eqref{eq:ux'}}{\geq} 2(n-\theta+\frac{\Delta(G)-1}{2})-d(x)\notag\\
&\stackrel{\eqref{eq:umin},\eqref{eq:Del}}{\geq} 2n-2\theta+m+n-1-(m+n-\theta)= 2n-1-\theta \label{eq:uu'}.
\end{align}
Let $u_1, \dots, u_n$ be distinct vertices in $X'$.  Now we choose distinct vertices $v_1, \dots, v_n\in V(G_x)\setminus \{u_1, \dots, u_n\}$ such that for all $i\in [n-1]$, $v_i\in N_{G_x}(u_i)\cap N_{G_x}(u_{i+1})$ and $v_n\in N_{G_x}(u_1)\cap N_{G_x}(u_{n})$.  Note that on the last step when we choose $v_n$, we have $|N_{G_x}(u_1)\cap N_{G_x}(u_{n})|\geq 2n-1-\theta\geq 2n-2$ and $|\{u_2, \dots, u_{n-1}, v_1, \dots, v_{n-1}\}|=2n-3$, so there exists a choice for $v_n$.  Thus $u_1v_1\dots u_nv_n$ is a cycle of length $2n$ in $G_x$ and we are done.  

\tbf{Case 2} ($|X'|=n-1$). By the choice of $x$ (consequently $X'$) and \eqref{eq:G_v'}, this implies that $|W|=1$ (i.e. $G_v$ is connected and $X$ is a leaf block), $w$ is adjacent to every vertex in $X\setminus \{w\}$, and every vertex in $X\setminus \{w\}$ has degree exactly $n-1$ in $V(G_v)\setminus\{w\}$ (and degree exactly $n$ in $V(G_v)$).  This in turn implies that $|X\cup Y|=d(v)=\Delta(G)=m+n$, $|Z|=m-\theta$, and every vertex in $X\setminus \{w\}$ is adjacent to every vertex in $Z$ (since $|N(x)\cap Z|=m+n-\theta-n=m-\theta= |Z|$).  So if $w$ has a neighbor $v_1\in I:=Z\setminus \{v\}$ we would be done because $G[X\setminus\{w\}, Z]$ induces a complete bipartite graph, $x$ has $n$ neighbors in $X$, and $w$ has 2 neighbors in $Z$.  Indeed, letting $u_1, \dots, u_{n-1}$ be the $n-1$ vertices in $X'$ and $v_2, \dots, v_{n-1}$ be an additional $n-2$ vertices in $I$, we have the cycle $vwv_1u_1v_2u_2\dots v_{n-1}u_{n-1}v$ of length $2n$ in $G_x$ and we are done.  So we have $N(w)\cap I=\emptyset$ (that is, $N(w)\cap Z=\{v\}$) and we are left with two cases.  

First suppose there exists $v_1v_2\in E(G[I])$.  Let $v_3,v_4,\dots,v_{n-1}$ be an additional
$n-3$ vertices in $I$, and let
$u_1,u_2,\dots,u_{n-1}$ be the $n-1$ vertices in
$X'$. Thus we have a cycle $vwu_1v_1v_2u_2v_3u_3\dots v_{n-1}u_{n-1}v$ of length $2n$ in $G_x$ and we are done.

So we finally suppose that there are no edges in $G[I]$.  So for all $z\in I$, $z$ has no neighbors in $Z\cup \{w\}$, thus $$m+n-\theta\leq |N(z)\cap (V(G_v)\setminus \{w\})|\leq N-(|Z|+1)=m+n-1=|V(G_v)\setminus \{w\}|$$ and thus $\theta=1$ and $G[Z,V(G_v)\setminus \{w\}]$ induces a complete bipartite graph.  By \eqref{eq:umin}, we have $d_{G_v}(w)=m+n-\theta-1=m+n-2=|V(G_v)|-2$ (since $w$ is adjacent to $v$) and thus there is a vertex $y\in Y$
such that $yw\notin E(G)$. By \eqref{eq:G_v'}, let
$u_1,u_2,\dots,u_{n}$ be $n$ vertices in $N_{G_v}(y)$ and
$v_1,v_2,\dots,v_{n}$ be $n$ vertices in $Z$.  Thus 
$v_1u_1v_2u_2\dots v_{n}u_{n}v_1$ is a cycle of length $2n$ in $G_y$ and we are done.
\end{proof}

\section{Even cycles versus even wheels}\label{sec:cycleWheel}

In this section we asymptotically determine $R(C_{2m}, W_{2n})$ for all sufficiently large $m$ and $n$.

\subsection{Lower bounds}\label{sec:3low}

\begin{example}
For all integers $m$ and $n$ with $m\geq n\geq 2$, we have 
\[
R(C_{2m}, W_{2n})\geq R(C_{2m}, F_n)\geq 4m-1.
\]
\end{example}

\begin{proof}
Take two disjoint red complete graphs of order $2m-1$ and put all blue edges between them.  There is clearly no red copy of $C_{2m}$ and since $\chi(W_{2n})=3=\chi(F_n)$ there is no blue copy of $F_n$ or $W_{2n}$.  
\end{proof}

In the case of $\frac{n}{2}\leq m< n$, we give a slight improvement of the lower bound which appears in \cite{YL}. 

\begin{example}
For all integers $m$ and $n$ with $1\leq \frac{n}{2}\leq m< n$, we have
\[
R(C_{2m}, W_{2n})\geq R(C_{2m}, F_n)\geq 2m+2n-2.
\]
\end{example}

\begin{proof}
Take disjoint red complete graphs of sizes $2m-1$, $n-1$, and $n-1$ and put all blue edges between them.  Since $n-1\leq 2m-1$, there is no red $C_{2m}$.  Since the blue graph is tripartite and every vertex has blue degree at most $n-1$ to one of the other parts, there is no blue matching of size $n$ in the blue neighborhood of any vertex and thus no blue $F_{n}$ (and consequently no blue $W_{2n}$).
\end{proof}

When $2\leq m<\frac{n}{2}$, our lower bound is essentially the same as the lower bound on $R(C_{2m}, K_{1,2n})$ from \cite{ALPZ}. 

\begin{example}
Let $m$ and $n$ be integers with $2\leq m< \frac{n}{2}$ and define $q$ to be the unique integer such that $\frac{2n-1}{q}< 2m-1\leq \frac{2n-1}{q-1}$.  Furthermore, if $\frac{2n-1}{q}< 2m-1\leq \frac{(q+1)(2n-1)}{q^2}$, let $r$ be the unique integer such that $1\leq r\leq q$ and $2n-1 \equiv r \bmod q$.

We have $$R(C_{2m}, W_{2n})\geq R(C_{2m}, K_{1,2n})\geq  \begin{cases} 
     q(2m-1)+1  & \frac{(q+1)(2n-1)}{q^2}<2m-1\leq \frac{2n-1}{q-1} \\
     (q+1)\ceiling{\frac{2n-1}{q}}-(q-r)  & \frac{2n-1}{q}< 2m-1\leq \frac{(q+1)(2n-1)}{q^2}
    \end{cases}.$$    
\end{example}

\begin{proof}
First suppose $\frac{(q+1)(2n-1)}{q^2}<2m-1\leq \frac{2n-1}{q-1}$.  Take $q$ many disjoint red complete graphs, each of size $2m-1$ and put blue edges between them.  There is clearly no red $C_{2m}$ and every vertex has blue degree at most $(q-1)(2m-1)\leq 2n-1$, so there is no blue $K_{1,2n}$.  

Next suppose $\frac{2n-1}{q}< 2m-1\leq \frac{(q+1)(2n-1)}{q^2}$.  
Let $r$ and $\ell$ be the unique integers such that $2n-1=q\ell+r$ where $1\leq r\leq q$.  Take $q+1$ disjoint red complete graphs, where $q-r+1$ of them have $\ceiling{\frac{2n-1}{q}}-1$ vertices and $r$ of them have $\ceiling{\frac{2n-1}{q}}$ vertices and add one vertex which is adjacent in red to every vertex in the sets of size $\ceiling{\frac{2n-1}{q}}-1$.  All other edges are blue.   So there are exactly $r\ceiling{\frac{2n-1}{q}}+(q-r+1)(\ceiling{\frac{2n-1}{q}}-1)+1=(q+1)\ceiling{\frac{2n-1}{q}}-(q-r)$ vertices in the graph.  Since $\frac{2n-1}{q}< 2m-1$, we have that $\ceiling{\frac{2n-1}{q}}\leq 2m-1$, so there is no red copy of $C_{2m}$.  Furthermore, every vertex has blue degree at most $$(q-r)(\ceiling{\frac{2n-1}{q}}-1)+r\ceiling{\frac{2n-1}{q}}=q\ceiling{\frac{2n-1}{q}}-(q-r)=q\ell+r=2n-1,$$ so there is no blue $K_{1,2n}$.   
\end{proof}

\subsection{Upper bounds}\label{sec:3up}

In order to prove Theorem \ref{thm:CycleWheel} and Theorem \ref{thm:m>n/2}, we need a number of results mentioned in Section \ref{sec:2up} as well as a few other auxiliary results.

\begin{theorem}[Faudree, Schelp \cite{FS}; Rosta \cite{Ros}]\label{thm:evenC}
For all integers $m,n\geq 2$ with $(m,n)\neq (2,2)$, we have
$$R(C_{2m}, C_{2n})=
 \begin{cases} 
     2m+n-1  & m\geq n \\
     m+2n-1  & m<n
    \end{cases}.
$$
\end{theorem}

\begin{theorem}[Bondy \cite{Bon}]\label{thm:bon}
If $G$ is a graph on $n\geq 3$ vertices with $\delta(G)\geq \frac{n}{2}$, then either $G$ is pancyclic or $G$ is isomorphic to $K_{n/2, n/2}$.
\end{theorem}

The following lemma appears in \cite{DW}, but as the proof is simple, we repeat it here.

\begin{lemma}\label{lem:partite_matching}
Let $G$ be a complete multipartite graph on $N$ vertices with parts $V_1, \dots, V_t$ such that $|V_1|\leq \dots \leq |V_t|$.
\begin{enumerate}
\item If $t\geq 3$ and $|V_t|\leq \frac{N}{2}$, then $G$ is pancyclic.
\item If $t\geq 3$ and $|V_t|>\frac{N}{2}$, then $G$ has a cycle of length $\ell$ for all $3\leq \ell\leq 2(N-|V_t|)$
\item If $t=2$, then $G$ has a cycle of length $2\ell$ for all $4\leq 2\ell\leq 2|V_1|$.
\end{enumerate}
\end{lemma}

\begin{proof}
(i) Suppose first that $t\geq 3$ and $|V_t|\leq \frac{N}{2}$.  In this case we have that $G$ is not bipartite and $\delta(G)\geq N-|V_t|\geq \frac{N}{2}$, so by Theorem \ref{thm:bon}, $G$ is pancyclic.

(ii) Let $V_t'\subseteq V_t$ with $|V_t'|=\sum_{i=1}^{t-1}|V_i|=N-|V_t|$ and then apply (i) to $G[V_1\cup \dots \cup V_{t-1}\cup V_t']$.

(iii) This follows directly since $G$ is a complete bipartite graph.
\end{proof}

\subsubsection{$m\geq n$}

We first prove an upper bound on $R(C_{2m}, W_{2n})$ in the case when $m\geq n$.  

\begin{proof}[Proof of Theorem \ref{thm:CycleWheel}]
Let $n$ and $m$ be integers such that $m\geq n\geq 2$ and set $c:=m-n$.  Let $N=\max\{4m+332-\floor{\frac{333}{251}c}, 4m-1\}$ and consider a 2-coloring of $K_N$.  We show that we can either find a red $C_{2m}$ or a blue $W_{2n}$.  
If $G_R$ is bipartite, then we have a blue clique $K$ on at least $\ceiling{N/2}$ vertices.  If $n=m$, then $\ceiling{N/2}>2m=2n$, and if $m>n$, then $\ceiling{N/2}\geq 2m>2n$; so either way, $K$ contains a blue $W_{2n}$.  So suppose that $G_R$ is not bipartite.

If there exists a vertex $v$ with $d_B(v)\geq 2m+n-1=3m-1-c$, then by Theorem \ref{thm:evenC}, we have a red $C_{2m}$ in $N_B(v)$ and we are done, or a blue $C_{2n}$ in $N_B(v)$ which gives a blue $W_{2n}$.  So we have 
\begin{equation}\label{eq:GR}
\delta(G_R)\geq N-1-(3m-2-c)\geq \frac{N}{4}+250>m,
\end{equation}
where the second inequality holds since $N\geq 4m+332-\floor{\frac{333}{251}c}\geq 4m+332-\frac{4c}{3}$.

If $G_R$ is 2-connected, then by \eqref{eq:GR} together with Theorem \ref{thm:dir} and Theorem \ref{thm:BFG} we have a red cycle of length exactly $2m$.  So suppose that $G_R$ is not 2-connected.  

If $G_R$ is separable, let $x$ be a vertex such that $G_R-x$ is disconnected and set $G_R':=G_R-x$. If $G_R$ is not connected, set $G_R':=G_R$.  Note that every component of $G_R'$ has order at least $\delta(G_R')+1\geq \delta(G_R)\geq \frac{N}{4}+250$, so in particular $G_R'$ has either two or three components.  If $G_R'$ has three components, then since all edges between these components are blue and by \eqref{eq:GR} all of the components have order greater than $m\geq n$, we have a blue copy of $W_{2n}$.  So suppose $G_R'$ has exactly two components $X_1$ and $X_2$ with $|X_1|\geq |X_2|> m\geq n$.  If there exists $i\in [2]$ and $v\in X_i$ such that $d_B(v, X_i)\geq n$, then we have a blue $W_{2n}$ (using the edges between $X_1$ and $X_2$).  So suppose 
\begin{equation}\label{eq:X1}
\text{for all $i\in [2]$, $\delta(G_R[X_i])\geq |X_i|-1-(n-1)=|X_i|-n$.}
\end{equation}

If $|X_1|\geq 2m$, then by \eqref{eq:X1} and the fact that $|X_1|\geq 2m\geq 2n$ we have $\delta(G_R[X_1])\geq |X_1|-n\geq \frac{|X_1|}{2}$ (with equality only if $|X_1|=2m=2n$). Thus by Theorem \ref{thm:bon}, $G'_R[X_1]$ is either pancyclic or isomorphic to $K_{m,m}$; in either case we get a red cycle of length $2m$.

Finally, suppose $|X_1|\leq 2m-1$.  Since $2m-1\geq |X_1|\geq \ceiling{\frac{|V(G_R')|}{2}}$, this implies that $|V(G_R')|=N-1=4m-2$ (which in turn implies $m\geq n+251$) and $|X_1|=2m-1=|X_2|$.  So we are in the case where $G_R$ is seperable and $x$ is a cut vertex of $G_R$.  If $x$ has blue degree at least $n$ to both $X_1$ and $X_2$, then we have a blue $W_{2n}$ as before using the blue edges between $X_1$ and $X_2$.  So suppose there exists $i\in [2]$ such that $x$ has blue degree at most $n-1$ to $X_i$ and thus red degree at least $2m-1-(n-1)=2m-n> m$ to $X_i$.  Now together with \eqref{eq:X1} we have a red graph $G_R[X_i\cup \{x\}]$ on $2m$ vertices with minimum degree greater than $m$, and thus by Theorem \ref{thm:bon}, $G_R[X_i\cup \{x\}]$ is pancyclic giving us a red cycle of length $2m$.
\end{proof}

\subsubsection{$\frac{n}{2}\leq m<n$}

Now we prove an upper bound on $R(C_{2m}, W_{2n})$ for sufficiently large $\frac{n}{2}\leq m<n$.

\begin{proof}[Proof of Theorem \ref{thm:m>n/2}]
Let $C':=75\cdot 10^4\cdot 6^5$ and $m_0:=45C'\cdot 6^4$ (the constants coming from Theorem \ref{thm:GHS} with $d=1/6$) and let $C:=3C'$.  Let $m$ and $n$ be positive integers with $m\geq m_0$ and $\frac{n}{2}\leq m<n$.  Let $N=2m+2n+C$ and consider an arbitrary 2-coloring of $K_N$.  We are looking for a red $C_{2m}$ or a blue $W_{2n}$.  If there is a vertex $v$ with blue degree at least $2n+m-1$, then by Theorem \ref{thm:evenC}, $N_B(v)$ either contains a red $C_{2m}$ and we are done, or a blue $C_{2n}$ which gives us a blue $W_{2n}$ with $v$ as the center.  So suppose that 
\begin{equation}\label{eq:mc1}
 \delta(G_R)\geq N-1-(2n+m-2)= 2m+2n+C-1-(2n+m-2)=m+C+1.
\end{equation}
By \eqref{eq:mc1} and the fact that $N=2m+2n+C\leq 6m+C$, we have $\delta(G_R)\geq m+C+1\geq \frac{1}{6}N+6$.  So by Lemma \ref{lem:2con} we can delete a set $X$ with $|X|\leq 4$ so that every component of $G_R':=G_R-X$ is 2-connected and has minimum degree at least $m+C+1-|X|\geq m+C-3\geq m+C'$.  

If some 2-connected component of $G_R'$ has at least $2m+C'$ vertices, then by Theorem \ref{thm:VZ} and Theorem \ref{thm:GHS} we can get a red cycle of length exactly $2m$.  So suppose that every 2-connected component of $G_R'$ has less than $2m+C'=2m+\frac{C}{3}$ vertices.  Since $2(2m+\frac{C}{3})+|X|<2m+2n+C$, it must be the case that $G_R'$ has at least three components.  Let $V_1,\dots, V_t$ be the vertex sets of such components and without loss of generality suppose $m+C'<|V_1|\leq |V_2|\leq |V_3|\leq \dots \leq |V_t|<2m+C'$.  

First suppose $t=3$.  If $n\leq |V_2|\leq |V_3|$, then since all edges between $V_i$ and $V_j$ are blue for $1\leq i<j\leq t$, we have a blue $W_{2n}$ with center in $V_1$. However, if $|V_1|\leq |V_2|<n$, then $|V_3|\geq 2n+2m+C-4-|V_1|-|V_2|>2m+C'$, a contradiction.

Finally suppose $t\geq 4$.  In this case we have $|V_2|+\dots+|V_{t-1}|>2m+2C'>|V_t|$ and $|V_2|+\dots+|V_t|\geq 2m+2n+C-4-|V_1|>2n$, so by Lemma \ref{lem:partite_matching} we have a blue copy of $W_{2n}$ with center in $V_1$ (since the blue complete multipartite graph with parts $V_2, \dots, V_t$ is pancyclic and has at least $2n$ vertices).
\end{proof}

\subsubsection{$2\leq m<\frac{n}{2}$}\label{sec:m<n/2}

In the case when $2\leq m<\frac{n}{2}$, the proof will be more complicated and we will need to make use of Szemer\'edi's regularity lemma as well as a version of the Edmonds-Gallai theorem for fractional matchings.  Given a graph $G$, we say that a set of edges $F\subseteq E(G)$ is a \emph{fractional matching} if every component of the subgraph induced by $F$ is an edge or an odd cycle (note that what we call a fractional matching is sometimes referred to in the literature as a \emph{2-matching}).

Given a graph $G$, disjoint subsets $A,B\subseteq V(G)$, and $\ep>0$, we define the density of $(X,Y)$ as $d(X,Y)=\frac{e(X,Y)}{|X||Y|}$ and we say that $(X,Y)$ is \emph{$\ep$-regular} if for all $X\subseteq A$ and $Y\subseteq B$ with $|X|\geq \ep |A|$ and $|Y|\geq \ep |B|$ we have $|d(X,Y)-d(A,B)|\leq \ep$.  

We will use Szemer\'edi's regularity lemma \cite{Sz} in the following form (see \cite[Exercise 2.1.23]{yZ} for example).

\begin{theorem}[Regularity lemma]\label{thm:reglem}
For all $\ep>0$ there exists $M$ such that for all $\ep\leq d\leq 1$ and all graphs $G$ on $n$ vertices there exists a partition $\{V_0, V_1, \dots, V_k\}$ of $V(G)$ such that 
\begin{enumerate}
\item $k\leq M$,
\item $|V_1|=\dots=|V_k|$ and $|V_0|\leq \ep n$,
\item for all $i\in [k]$, $|\{j\in [k]: (V_i, V_j) \text{ is not $\ep$-regular}\}|<\ep k$.
\end{enumerate}
\end{theorem}

% \begin{theorem}[Degree form of the regularity lemma \cite{Sz}]\label{thm:reglem1}
% For all $\ep>0$ there exists $M$ such that for all $\ep\leq d\leq 1$ and all graphs $G$ on $n$ vertices there exists a partition $\{V_0, V_1, \dots, V_k\}$ of $V(G)$ and a spanning subgraph $G'\subseteq G$ such that 
% \begin{enumerate}
% \item $k\leq M$,
% \item $|V_1|=\dots=|V_k|$ and $|V_0|\leq \ep n$,
% \item $d_{G'}(v)> d_G(v)-(d+\ep)n$ for all $v\in V(G)$,
% \item $e(G'[V_i])=0$ for all $i\in [k]$, and
% \item for all $1\leq i<j\leq k$, $G'[V_i, V_j]$ is $\ep$-regular with density 0 or density greater than $d$. 
% \end{enumerate}
% \end{theorem}

For our application, we will begin with a 2-colored graph $G$.  We will apply the regularity lemma to the graph induced by the edges of one of the colors, say red, but we will need the reduced graph to account for the blue edges as well.  So we will tailor our definition of the reduced graph to this specific context. Let $\ep, d\in \RR^+$, let $G$ be a 2-colored graph, and let $\{V_0, V_1, \dots, V_k\}$ be a partition of $V(G)$. The \emph{2-colored reduced graph} of $G$ with parameters $\ep$ and $d$ is the 2-colored graph $\Gamma$ with vertex set $V(\Gamma)=\{V_1, \dots, V_k\}$ where $\{V_i, V_j\}\in E(\Gamma_R)$ if and only if $(V_i, V_j)$ is $\ep$-regular with density greater than $d$ in $G_R$, and $\{V_i, V_j\}\in E(\Gamma_B)$ if and only if $d_B(V_i, V_j)\geq 1-d$.  

\begin{lemma}[Properties of the 2-colored reduced graph]\label{lem:reduced}
Let $c>0$ and let $G$ be a 2-colored complete graph on $n$ vertices such that $\delta(G_R)\geq cn$.  Let $\ep, d\in \RR$ such that $0<2\ep\leq d<\frac{c}{2}$.  Apply Theorem \ref{thm:reglem} to $G_R$ to get an $\ep$-regular partition $\{V_0, V_1, \dots, V_k\}$ of $G_R$ and let $\Gamma$ be the resulting 2-colored reduced graph with parameters $\ep$ and $d$.  We have 
\begin{enumerate}
\item $\delta(\Gamma)\geq (1-\ep)k$ and
\item $\delta(\Gamma_R)\geq (c-d-\ep)k$.
\end{enumerate}
\end{lemma}

\begin{proof}
Note that (i) follows from Theorem \ref{thm:reglem}(iii) and the fact that if $d_R(V_i, V_j)\leq d$, then $d_B(V_i, V_j)\geq 1-d$.  As for (ii), let $i\in [k]$ and let $B=\{j\in [k]: d_R(V_i, V_j)\leq d \text{ or } (V_i, V_j) \text{ is not $\ep$-regular}\}$.  We have $$|V_i||V_j||B|\leq e(V_i, \bigcup_{j\in B}V_j)\leq \ep k|V_i||V_j|+d|V_i||V_j|k$$ and thus $|B|\leq (\ep+d)k$.  This implies $\delta(\Gamma_R)\geq (c-d-\ep)k$.  
\end{proof}

The following lemma is a fractional matching version of {\L}uczak's connected matching lemma \cite{Luc} (see the discussion in \cite[Section 1.1]{Letz}).

\begin{lemma}\label{lem:frac_cycle}
Let $0<\frac{1}{n_0}\ll \ep\ll \eta\leq 1$.  Let $G$ be a graph on $n\geq n_0$ vertices and let $\Gamma$ be the resulting reduced graph after an application of Theorem \ref{thm:reglem} to $G$ with parameters $\ep$ and $d$.  If $\Gamma$ contains a connected subgraph $\Gamma'$ such that $\Gamma'$ contains a fractional matching covering at least $\eta |V(\Gamma)|$ vertices, then $G$ contains a cycle of length $2m$ for all $4\leq 2m\leq (\eta-\ep)n$.
\end{lemma}

We also need the following lemma to deal with the case when we can't find a red fractional matching of the desired size in the reduced graph. 

\begin{theorem}[Pulleyblank \cite{Pull}]\label{thm:frac_GE}
If the largest fractional matching in a graph $G$ covers $p$ vertices, then there exists disjoint sets $A,C,D\subseteq V(G)$ such that\footnote{Strictly speaking, $\{A, C, D\}$ may not be a partition since we could have that $A$, $C$, or $D$ is empty.} $A\cup C\cup D=V(G)$ and 
\begin{enumerate}
\item $|D|=|A|+|V(G)|-p$,
\item $|A|\geq \delta(G)$ and $2|A|+|C|=p$, so in particular $|A|\leq \floor{p/2}$ and $|C|\leq p-2\delta(G)$, and
\item $D$ is an independent set and there are no edges between $C$ and $D$.  
%\item $|C|+|D|=|V(G)|-|A|$
\end{enumerate}
\end{theorem}

We will also use the following well-known result of Moon and Moser \cite{MM}.

\begin{theorem}[Moon and Moser]\label{thm:MM}
Let $n$ be an integer with $n\geq 2$ and let $G$ be a bipartite graph with bipartition $\{A, B\}$ such that $|A| = |B| = n$. If $\delta(G) > n/2$, then $G$ has a Hamiltonian cycle. 
\end{theorem}

%We also use the following corollary of Theorem \ref{thm:jackson}.

% \begin{corollary}[Jackson]\label{cor:jackson}
% Let $n \geq 2$ be an integer and let $G$ be a bipartite graph with bipartition $\{X,Y\}$ such that $|X| \geq n$, $|Y| \geq n$. If every vertex in $Y$ has degree at least $\max\{n, |X| - (n - 2)\}$ to $X$, then $G$ contains a cycle of length $2n$.
% \end{corollary}

% \begin{proof}
% We consider two cases based on the size of $X$.

% \textbf{Case 1} ($|X| \leq 2n-2$) In this case $\max\{n, |X| - n + 2\}=n$.  Since every vertex in $Y$ has degree at least $n$ to $X$, we can apply Theorem \ref{thm:jackson} to $G$ with $k=n$ and $Y'\subseteq Y$ with $|Y'|=n$ to get a cycle of length $2n$.

% \textbf{Case 2} ($|X| > 2n-2$.) In this case $\max\{n, |X| - (n-2)\}=|X|-(n-2)$. Let $X' \subseteq X$ with $|X'|=2n-2$. For all $y \in Y$, the number of neighbors of $y$ in $X'$ is at least $|X'| - (n-2) = n$.  Thus we can then apply Theorem \ref{thm:jackson} to the induced subgraph $G[X', Y]$ with $k=n$ and $Y'\subseteq Y$ with $|Y'|=n$ to get a cycle of length $2n$.
% \end{proof}

The following is a generalization of Lemma \ref{lem:partite_matching} where we allow for a nearly complete multipartite graph (and it is stated directly in terms of wheels).

\begin{lemma}\label{lem:case1}
Let $t$ be an integer with $t\geq 3$ and let $0<\frac{1}{n_0}\ll d\ll \alpha\ll \beta\leq \frac{1}{t}$.  Let $n$ and $N$ be positive integers with $n\geq n_0$ and $N\leq 3n$.
Let $G$ be a $t$-partite graph on $N$ vertices with partition $\{V_1, \dots, V_t\}$ such that $\beta N\leq |V_1|\leq \dots \leq |V_t|$.  If 
\begin{enumerate}
\item $|V_2|+\dots+|V_t|\geq (2+\beta)n$, 
\item $|V_2|+\dots+|V_{t-1}|\geq (1+\beta)n$, and
\item for all $1\leq i<j\leq t$, $e(V_i, V_j)\geq (1-d)|V_i||V_j|$,
\end{enumerate}
then $W_{2n}\subseteq G$.
\end{lemma}

\begin{proof}  
For all distinct $i,j\in [t]$, let $X_{i,j}=\{x\in V_i: d(x, V_j)<(1-\alpha)|V_j|\}$ and note that since $(1-d)|V_i||V_j|\leq e(V_i, V_j)<|X_{i,j}|(1-\alpha)|V_j|+(|V_i|-|X_{i,j}|)|V_j|$ we have $|X_{i,j}|<\frac{d}{\alpha}|V_i|$.  For all $i\in [t]$, let $X_i=\bigcup_{j\in [t]\setminus \{i\}}X_{i,j}$ and note that $|X_i|<\frac{d (t-1)}{\alpha}|V_i|$.  Let $G'=G[V(G)\setminus \bigcup_{i\in [t]}X_i]$ and for all $i\in [t]$, let $V_i'=V_i\setminus X_i$.  So for all $i\in [t]$ we have 
\begin{equation}\label{eq:Vi'}
|V_i'|=|V_i|-|X_i|>(1-\frac{d (t-1)}{\alpha})|V_i|\geq (1-\alpha)|V_i|,
\end{equation}
and for all distinct $i,j\in [t]$, we have 
\begin{equation}\label{eq:ViVj}
\delta_{G'}(V_i', V_j')\geq (1-\alpha)|V_j|-|X_j|\geq (1-\alpha-\frac{d(t-1)}{\alpha})|V_j| \geq (1-2\alpha)|V_j'|,
\end{equation}
where the last inequality in each case holds since $t\leq \frac{1}{\beta}$ and $d< \alpha^2 \beta$.

Let $v\in V_1'$ and for all $2\leq i\leq t$, let $V_i''=N(v)\cap V_i'$.  By reindexing if necessary, suppose $|V_2''|\leq \dots \leq |V_t''|$.  Set $H := G'[N(v)]$, $N_H:=|V(H)|$, and $X:=\bigcup_{i=2}^{t-1}V_i''$.  By \eqref{eq:Vi'} and \eqref{eq:ViVj} we have 
\begin{equation}\label{eq:N_H1}
N_H=\sum_{i=2}^t|V_i''|\geq (1-2\alpha)\sum_{i=2}^t|V_i'|\geq (1-2\alpha)(1-\alpha)(2+\beta)n\geq (2+\frac{\beta}{2})n
\end{equation}
and 
\begin{equation}\label{eq:|X|}
|X|=\sum_{i=2}^{t-1}|V_i'|\geq (1-2\alpha)\sum_{i=2}^{t-1}|V_i'|\geq (1-2\alpha)(1-\alpha)(1+\beta)n\geq (1+\frac{\beta}{2})n.  
\end{equation}
Now we check that $H$ contains a cycle of length $2n$ which will give us a $W_{2n}$ with $v$ as the center.  First note that by \eqref{eq:N_H1}, we have $|V(H)|>2n$ and thus $H$ is large enough to contain the desired cycle.  

\textbf{Case 1:} $|V_t''|\leq (\frac{1}{2}-2\alpha)N_H$.  In this case we have that for all $2\leq i\leq t$ and all $u\in V_i''$, 
\begin{align*}
d_H(u)\geq (1-2\alpha)(N_H-|V_i''|)\geq (1-2\alpha)(N_H-(\frac{1}{2}-2\alpha)N_H) &= (1-2\alpha)(\frac{1}{2}+2\alpha)N_H\\
&= (\frac{1}{2}+\alpha-4\alpha^2)N_H>\frac{N_H}{2}.
\end{align*}
Thus by Theorem \ref{thm:bon}, we have a copy of $C_{2n}$ in $H$.

\textbf{Case 2:} $|V_t''|> (\frac{1}{2}-2\alpha)N_H$.  Note that $$|V_t''|>(\frac{1}{2}-2\alpha)N_H \geq (\frac{1}{2}-2\alpha)(2+\frac{\beta}{2})n = (1+\frac{\beta}{4}-4\alpha-\alpha\beta)n> (1+\frac{\beta}{8})n.$$ By \eqref{eq:|X|} we have $|X|>(1+\frac{\beta}{2})n$.  Furthermore, every vertex in $X$ has degree at least $(1-2\alpha)|V_t''|$ to $V_t''$ and 
every vertex in $V_t''$ has degree at least $(1-2\alpha)|X|$ to $X$.  
Let $A\subseteq X$ and $B\subseteq V_t''$ with $|A|=n=|B|$.
Since every vertex in $A$ has at least $|B|-2\alpha|V_t''|>\frac{n}{2}$ neighbors in $B$ and every vertex in $B$ has at least $|A|-2\alpha|X|>\frac{n}{2}$ neighbors in $A$, we can apply Theorem  \ref{thm:MM} to the bipartite subgraph $G'[A,B]$ to get a copy of $C_{2n}$ in $H$.
\end{proof}

The following lemma will be used in the case when we cannot use the previous lemma.  This will happen when there are parts of the partition which are very large in terms of $m$ and in this case we will apply Theorem \ref{thm:frac_GE} to the large parts, giving us extra information to work with.  To prove the lemma, we will use the following Chv\'atal-type strengthening of Theorem \ref{thm:bon}. 

\begin{theorem}[Schmeichel, Hakimi \cite{SH}]\label{SHthm}
Let $G$ be a graph on $n\geq 3$ vertices with degree sequence $d_1\leq d_2\leq \dots \leq d_n$.  If for all $1\leq k< n/2$, $d_k\leq k$ implies $d_{n-k}\geq n-k$, then $G$ is pancyclic or isomorphic to $K_{n/2,n/2}$.

In particular, if for all $1\leq k< n/2$, $d_k\leq k$ implies $d_{n-k}>n-k$, then $G$ is pancyclic.
\end{theorem}

\begin{lemma}\label{lem:case2}
Let $t$ be a positive integer and let $0 < \frac{1}{n_0} \ll d \ll \eta\ll \alpha \ll \beta \ll \gamma \leq \frac{1}{t}$. Let $n, m,$ and $N$ be positive integers with $n \geq n_0$, $\frac{\alpha}{\gamma} n \leq m < \frac{n}{2}$, and $(2+\gamma)n+m \leq N \leq 3n$. Let $G^*$ be a graph on $N$ vertices with partition $\{V_1, \dots, V_t\}$ such that $\beta N \leq |V_1| \leq \dots \leq |V_t|$ and 
\begin{enumerate}
\item for all $1 \leq i < j \leq t$, $G^*[V_i, V_j]$ is a complete bipartite graph,
\item $|V_{t-r}| < (2+\beta)m \leq |V_{t-r+1}|$,
\item for all $t-r+1 \leq i \leq t$, $V_i$ is partitioned as $\{A_i, C_i, D_i\}$ where $|C_i| + 2|A_i| < (2+\alpha)m$ and $|D_i| \geq |A_i| + |V_i| - (2+\alpha)m$, and 
\item for all $t-r+1 \leq i \leq t$, $G^*[D_i]$ is complete and $G^*[C_i, D_i]$ is a complete bipartite graph.
\end{enumerate}
If $G \subseteq G^*$ with $V(G)=V(G^*)$ such that for all $i,j \in [t]$ (not necessarily distinct), $e_{G}(V_i, V_j) \geq (1-d)e_{G^*}(V_i, V_j)$, then $G-A_t$ contains a copy of $W_{2n}$ with center in $D_t$.
\end{lemma}

\begin{figure}[ht]
    \centering
    \includegraphics[scale=.8]{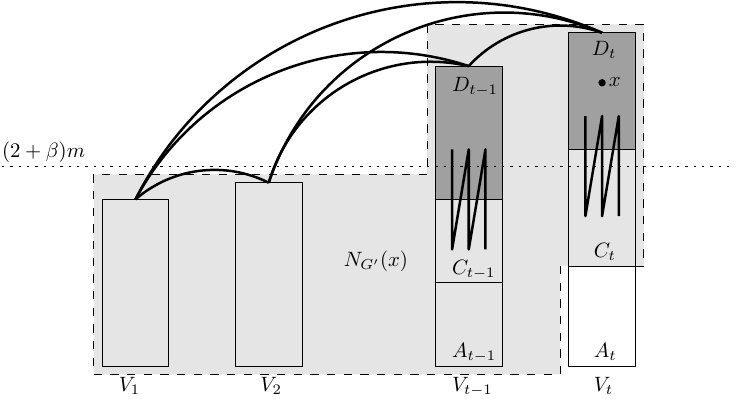}
    \caption{Proving that there is a copy of $W_{2n}$ having its center in $D_t$.  Note that for all $1\leq i<j\leq t$, we have all possible edges between $V_i$ and $V_j$ in the original graph $G^*$.}
    \label{fig:lem3-2}
\end{figure}

\begin{proof}
For all $i, j \in [t]$ (not necessarily distinct), let $X_{i,j} = \{x \in V_i :  d_G(x, V_j) < d_{G^*}(x, V_j)- \eta|V_j|\}$. 
Since 
\begin{align*}
(1-d)e_{G^*}(V_i, V_j)\leq e_{G}(V_i, V_j) &<  \sum_{x\in X_{i,j}}(d_{G^*}(x, V_j)- \eta|V_j|)+\sum_{v\in V_i\setminus X_{i,j}}d_{G^*}(v, V_j)\\
&=e_{G^*}(V_i, V_j)-|X_{i,j}|\eta |V_j|,
\end{align*}
we have $|X_{i,j}|<  \frac{d}{\eta |V_j|}e_{G^*}(V_i, V_j)\leq \frac{d}{\eta}|V_i|$. For all $i \in [t]$, let $X_i = \bigcup_{j=1}^t X_{i,j}$. Let $G' = G[V(G) \setminus \bigcup_{i=1}^t X_i]$ and for all $i \in [t]$, let $V_i' = V_i \setminus X_i$. Note that $|X_i| < \frac{dt}{\eta}|V_i| \leq \eta|V_i|$ (since $t \leq \frac{1}{\beta}$ and $d \ll \eta$), so $|V(G')| \geq N - \eta N \geq N - 3\eta n$. Furthermore, for any $v \in V(G')$, $d_{G'}(v) \geq d_{G^*}(v) - \sum_{j=1}^t 2\eta|V_j| \geq d_{G^*}(v) - 6\eta n$. For every set $S\subseteq V(G)$ that is a union of parts or subsets of parts (such as $A_i, C_i,$ or $D_i$), we write $S'=S\cap V(G')$ and we note that $|S'| \geq |S| - 3\eta n$.

Note that by (ii) we have $|V_t|\geq (2+\beta)m$ and thus by (iii), $|D_t|\geq(\beta-\alpha)m$ which implies that $|D_t'|\geq |D_t|-\eta|V_t|\geq (\beta-2\alpha)m$.  Choose a vertex $x\in D_t'$, set $H = G'[N_{G'}(x)\setminus A_t]$, and set $N_H = |V(H)|$.  We will show that $N_H>2n$ and $H$ satisfies the conditions of Theorem \ref{SHthm} for being pancyclic.  Thus we will get a copy of $C_{2n}$ in $H$ which gives us a copy of $W_{2n}$ with $x$ as the center (see Figure \ref{fig:lem3-2}).  For all $S'\subseteq V(G')$, we write $S''=S'\cap V(H)$. Because $x$ is complete to $G^* \setminus A_t$, the number of non-neighbors of $x$ in $G' \setminus A_t$ is bounded by $6\eta n$. Thus, we have
\begin{equation}\label{eq:S''}
|S''|\geq |S'|- 6\eta n \geq |S|- 9\eta n.
\end{equation}

First note that by (iii) we have
\begin{equation}\label{eq:N_H}
N_H \ge N - 9\eta n - |A_t|\geq ((2+\gamma)n+m) - 9\eta n -(1+\frac{\alpha}{2})m > (2+\frac{\gamma}{2})n.
\end{equation}
For all $v \in V(H)$, let $M(v)$ denote the set of non-neighbors of $v$ in the original graph $G^*$. Thus we have
\begin{equation}\label{eq:dH}
d_H(v) \geq N_H - 6\eta n - |M(v)|.
\end{equation}
Note that if $|M(v)|<(2+\beta)m$, then since $m<\frac{n}{2}$, $$|M(v)|+6\eta n<(2+\beta)m+6\eta n<(1+\frac{\beta}{2})n+6\eta n<(1+\frac{\gamma}{4})n\leq \frac{N_H}{2}$$ and thus by \eqref{eq:dH} we have $d_H(v)>\frac{N_H}{2}$.

\begin{claim}\label{clm:NH/2}
For all $v\in V(H)\setminus \bigcup_{i=t-r+1}^{t-1} A_i''$, $|M(v)|<(2+\beta)m$; consequently, $d_H(v)>\frac{N_H}{2}$.
\end{claim}

\begin{proofclaim}
Let $v \in V(H) \setminus \bigcup_{i=t-r+1}^{t-1} A_i''$.  Note that if $v\in V_j$, then $M(v)\subseteq V_j$. So if $v \in V_i''$ for $i \leq t-r$, then $|M(v)| \leq |V_i| \leq |V_{t-r}| < (2+\beta)m$ by (ii). If $v \in C_i'' \cup D_i''$ for $t-r+1 \leq i \leq t$, then because $D_i$ is complete to $C_i \cup D_i$ in $G^*$, we have $M(v)\subseteq A_i \cup C_i$.  So by (iii), $|M(v)| \leq |C_i| + |A_i| \leq (2+\alpha)m < (2+\beta)m$.
\end{proofclaim}

If for all $t-r+1 \leq i \leq t-1$ and $v \in A_i''$, we have $d_H(v) > \frac{N_H}{2}$, then together with Claim \ref{clm:NH/2} this means that all vertices in $H$ have degree strictly greater than $\frac{N_H}{2}$ and thus by Theorem \ref{SHthm} (in fact, Theorem \ref{thm:bon} is sufficient in this case), $H$ is pancyclic. By \eqref{eq:N_H}, we have $N_H>2n$ and thus $H$ contains a copy of $C_{2n}$ which gives us a $W_{2n}$ with $x$ as the center.  

So suppose not; that is, suppose there exists $t-r+1 \leq i \leq t-1$ and a vertex $v \in A_i''$ such that $d_H(v) \leq \frac{N_H}{2}$.  We will show that in this case there is exactly one such index $i$.  We have $M(v)\subseteq V_i$ and thus $|M(v)| \leq |V_i|$. Therefore, by \eqref{eq:dH} we have $\frac{N_H}{2} \geq d_H(v) \geq N_H - |V_i| - 6\eta n$ which implies 
\begin{equation}\label{eq:ViNH1}
|V_i| \geq \frac{N_H}{2} - 6\eta n.
\end{equation}
Now suppose for contradiction that we have $t-r+1\leq i<j\leq t-1$ such that $|V_t|\geq |V_j|\geq |V_i| \geq \frac{N_H}{2} - 6\eta n$. By \eqref{eq:N_H} and the bounds involving $n$, $m$, $N$ in the hypotheses, this implies
\begin{align*}
N \geq |V_i| + |V_j| + |V_t| \stackrel{\mathclap{\eqref{eq:ViNH1}}}{\geq} 3(\frac{N_H}{2} - 6\eta n)&\geq \frac{3}{2}(N - 9\eta n - |A_t|) - 18\eta n \\
&= N + \frac{1}{2}N - \frac{3}{2}|A_t| - \frac{63}{2}\eta n \\
&\geq N + \frac{5+2\gamma}{2}m - \frac{3}{2}(1+\frac{\alpha}{2})m - \frac{63}{2}\eta \frac{\alpha}{\gamma}m \\
&= N + (1 + \gamma - \frac{3\alpha}{4} - \frac{63\eta \alpha}{2\gamma})m > N,
\end{align*}
a contradiction (since $\alpha, \eta\ll \gamma$). Thus, there is exactly one such index $i$.

We now verify that $H$ satisfies the conditions of Theorem \ref{SHthm} for being pancyclic. Let $d_1 \leq d_2 \leq \dots \leq d_{N_H}$ be the degree sequence of $H$. Suppose there exists an integer $k < \frac{N_H}{2}$ such that $d_k \leq k$. We must show that $d_{N_H-k} > N_H - k$. As established, the only vertices in $H$ that can have degree less than $\frac{N_H}{2}$ belong to $A_i''$. Because $d_k \leq k < \frac{N_H}{2}$, all of these vertices corresponding to $d_1, \dots, d_k$ must be in $A_i''$, implying $k \leq |A_i''| \leq |A_i|$. 

Let $u \in A_i''$ be the vertex with degree $d_H(u) = d_k \leq k$. As usual we have $|M(u)| \leq |V_i|$. So by \eqref{eq:dH}, $d_k \geq N_H - |V_i| - 6\eta n$. Because $d_k \leq k < \frac{N_H}{2}$, this gives $\frac{N_H}{2} > N_H - |V_i| - 6\eta n$ which implies 
\begin{equation}\label{eq:ViNH2}
|V_i| > \frac{N_H}{2} - 6\eta n.
\end{equation}
Note that by \eqref{eq:S''} we have 
\begin{align*}
|D_i''| \geq |D_i| - 9\eta n \geq |A_i| + |V_i| - (2+\alpha)m - 9\eta n &\stackrel{\mathclap{\eqref{eq:ViNH2}}}{>} |A_i| + \left(\frac{N_H}{2} - 6\eta n\right) - (2+\alpha)m - 9\eta n\\
&>|A_i| + \left((1+\frac{\gamma}{4})n\right) - (2+\alpha)\frac{n}{2} - 15\eta n\\
&=|A_i|+(\frac{\gamma}{4}-15\eta-\frac{\alpha}{2})n>|A_i|.
\end{align*}
So provided the degrees of the vertices in $D_i''$ are large enough we will be done.

For all $v \in D_i''$ we have $|M(v)| \leq |A_i|$, and thus by \eqref{eq:dH} and the fact that $i<t$ (which implies $|V_i|\leq \frac{N}{2}$) we have
\begin{align*}d_k + d_H(v) &\geq (N_H - |V_i| - 6\eta n) + (N_H - |A_i| - 6\eta n) \\
&= N_H + N_H - |V_i| - |A_i| - 12\eta n \\
&\stackrel{\mathclap{\eqref{eq:N_H}}}{\geq} N_H + N - 9\eta n - |A_t| - \frac{N}{2} - |A_i| - 12\eta n\\
&= N_H + \frac{N}{2} - |A_t| - |A_i| - 21\eta n \\
&\geq N_H + \frac{N}{2} -(2+\alpha)m - 21\eta n \\
&\geq N_H + (\frac{\gamma}{2}-\frac{\alpha}{2}-21\eta)n> N_H,
\end{align*} 
where the last two inequalities hold by the bounds involving $n$, $m$, $N$ and the fact that $\alpha, \eta\ll \gamma$.  Thus $d_k+d_{N_H-k} > N_H$, which implies that $H$ is pancyclic. By \eqref{eq:N_H} we have $N_H > 2n$, and thus $H$ contains a cycle of length $2n$ giving us a copy of $W_{2n}$ in $G - A_t$ with $x$ as the center.
\end{proof}

Now we proceed with showing 
$R(C_{2m}, W_{2n})= 
\begin{cases}
(2q+o(1))m, & \frac{(q+1)n}{q^2}\leq m<\frac{n}{q-1} \\
(2+\frac{2}{q}+o(1))n, & \frac{n}{q}\leq m<\frac{(q+1)n}{q^2}
\end{cases}.$

\begin{proof}[Proof of Theorem \ref{thm:m<n/2} in the case $\frac{(q+1)n}{q^2}\leq m<\frac{n}{q-1}$]
Let $\gamma>0$, let $m_0\gg \frac{1}{\gamma}$ and let $n$, $m$ be integers with $m_0\leq m<\frac{n}{2}$ and $\frac{(q+1)n}{q^2}\leq m<\frac{n}{q-1}$, where $q$ is the unique integer such that $\frac{n}{q}\leq m< \frac{n}{q-1}$. Let $G$ be a 2-colored complete graph on $N=(2q+\gamma)m$ vertices. If there is a vertex $v$ with blue degree at least $m+2n-1$, then by Theorem \ref{thm:evenC} we have a red $C_{2m}$ or blue $C_{2n}$ in $N_B(v)$ and we are done.  Since $n\leq \frac{q^2}{q+1}m$, we have
\begin{equation}\label{eq:dG1}
\delta(G_R)\geq N-1-(m+2n-2)\geq (2q-1+\gamma)m-2n\geq (1-\frac{2}{q+1}+\gamma)m=\frac{\frac{q-1}{q+1}+\gamma}{2q+\gamma}N. 
\end{equation}
Apply Theorem \ref{thm:reglem} to $G_R$ to obtain an $\ep$-regular partition $\{V_0, V_1, \dots, V_k\}$ of $G_R$ with $0<\ep \ll \gamma$. Let $\Gamma$ be the resulting 2-colored reduced graph with parameters $\ep$ and $d$. By \eqref{eq:dG1} and Lemma \ref{lem:reduced} we have 
\begin{equation}\label{eq:dGam1}
\delta(\Gamma_R)\geq (\frac{\frac{q-1}{q+1}+\gamma}{2q+\gamma}-d-\ep)k\geq \frac{\gamma}{3q} k.
\end{equation}
Let $\Gamma_R^1, \dots, \Gamma_R^t$ be the connected components of $\Gamma_R$ and note that by \eqref{eq:dGam1} all of these components have more than $\frac{\gamma}{3q}k$ vertices.  For all $i\in [t]$, let $U_i=\bigcup_{V_j\in V(\Gamma_R^i)}V_j$. Without loss of generality suppose $\frac{\gamma}{4q}N\leq |U_1|\leq \dots \leq |U_t|.$

\textbf{Case 1} ($|U_t|< (2+\frac{\gamma}{q})m$).  In this case we show that the conditions of Lemma \ref{lem:case1} are satisfied in $G_B[U_2\cup \dots \cup U_t]$ which will give us a blue copy of $W_{2n}$ with center in $U_1$.

Since $N=(2q+\gamma)m=q(2+\frac{\gamma}{q})m$, then by the case we must have $t\geq q+1\geq 4$. Consequently, $|U_1|\leq \frac{N}{q+1}$ and 
\begin{equation}\label{eq:2tot}
|U_2|+\dots+|U_t|\geq \frac{qN}{q+1} = (\frac{2q+\gamma}{q+1})qm \geq (2+\frac{\gamma}{q})n,
\end{equation}
where the last inequality holds since $m\geq \frac{(q+1)n}{q^2}$.  Furthermore, by \eqref{eq:2tot} and the fact that $m < \frac{n}{q-1}$, we have
\begin{align*}
|U_2|+\dots+|U_{t-1}| \geq (2+\frac{\gamma}{q})n - |U_t| > (2+\frac{\gamma}{q})n - (2+\frac{\gamma}{q})\frac{n}{q-1}= \Big(1 + \frac{q-3}{q-1} + \frac{\gamma(q-2)}{q(q-1)}\Big)n.
\end{align*}
Thus Lemma \ref{lem:case1} is satisfied with $\beta = \min\{\frac{\gamma}{q}, \frac{q-3}{q-1} +\frac{\gamma(q-2)}{q(q-1)}\}$ (note that $q-3\geq 0$), giving us a blue $W_{2n}$ in $G_R[U_1\cup \dots \cup U_t]$ with center in $U_1$.

\textbf{Case 2} ($|U_{t}|\geq (2+\frac{\gamma}{q})m$). For all $i\in [t]$ with $|U_i|\geq (2+\frac{\gamma}{q})m$, if $\Gamma_R^i$ contains a fractional matching covering at least $\frac{(2+\frac{\gamma}{q})m}{N}k$ vertices of $\Gamma_R^i$, then by Lemma \ref{lem:frac_cycle} we get a red copy of $C_{2m}$. So suppose for all such $i$ there is no such fractional matching, and apply Theorem \ref{thm:frac_GE} to $\Gamma_R^{i}$ to get a partition $\{\mathcal{A}_{i}, \mathcal{C}_{i}, \mathcal{D}_{i}\}$ of $\Gamma_R^{i}$ which gives a resulting partition $\{A_i, C_i, D_i\}$ of $U_i$. By Theorem \ref{thm:frac_GE}, $\mathcal{D}_i$ is an independent set in $\Gamma_R$, and there are no edges between $\mathcal{C}_i$ and $\mathcal{D}_i$ in $\Gamma_R$. Furthermore, there are no edges in $\Gamma_R$ between distinct components $\Gamma_R^i$ and $\Gamma_R^j$. By Lemma \ref{lem:reduced}(i), we have $\delta(\Gamma)\geq (1-\ep)k$ and thus all but an $\ep$-proportion of non-edges in $\Gamma_R$ must be edges in $\Gamma_B$. Since an edge in $\Gamma_B$ implies a blue edge density of at least $1-d$ between the corresponding clusters, the conditions of Lemma \ref{lem:case2} are satisfied in $G_B[U_1\cup \dots \cup U_t]$ with $\beta=\frac{\gamma}{q}$.  Note that since $m \ge \frac{(q+1)n}{q^2}$ and $q \ge 3$, we have $$N-m=(2q-1+\gamma)m\geq (2q-1+\gamma)\frac{(q+1)n}{q^2} =(2+\frac{q-1}{q^2}+\gamma\frac{q+1}{q^2})n.$$ Thus the lower bound on $N$ required for Lemma \ref{lem:case2} is satisfied with $\frac{q-1+\gamma(q+1)}{q^2}$ in place of $\gamma$.  Applying Lemma \ref{lem:case2} to $G_B[U_1\cup \dots\cup U_t]$ gives a blue $W_{2n}$.
\end{proof}

\begin{proof}[Proof of Theorem \ref{thm:m<n/2} in the case $\frac{n}{q}\leq m<\frac{(q+1)n}{q^2}$]
Let $\gamma>0$, let $m_0\gg \frac{1}{\gamma}$ and let $n$, $m$ be integers with $m_0\leq m<\frac{n}{2}$ and $\frac{n}{q}\leq m<\frac{(q+1)n}{q^2}$, where $q$ is the unique integer such that $\frac{n}{q}\leq m< \frac{n}{q-1}$. Let $G$ be a 2-colored complete graph on $N=(2+\frac{2}{q}+\gamma)n$ vertices. Note that since $n \leq qm$, we have 
\begin{equation}\label{eq:nqm}
N \leq (2q+2+q\gamma)m. 
\end{equation}

If there is a vertex $v$ with blue degree at least $m+2n-1$, then by Theorem \ref{thm:evenC} we have a red $C_{2m}$ or blue $C_{2n}$ in $N_B(v)$ and we are done. Since $m< \frac{(q+1)n}{q^2}$ in this case, 
\begin{equation}\label{eq:dG2}
\delta(G_R)\geq N-1-(m+2n-2)\geq (\frac{2}{q}+\gamma)n-m> (\frac{1}{q}-\frac{1}{q^2}+\gamma)n=\frac{\frac{1}{q}-\frac{1}{q^2}+\gamma}{2+\frac{2}{q}+\gamma}N.
\end{equation}
Apply Theorem \ref{thm:reglem} to $G_R$ to obtain an $\ep$-regular partition $\{V_0, V_1, \dots, V_k\}$ of $G_R$ with $0<\ep \ll \gamma$.
Let $\Gamma$ be the resulting 2-colored reduced graph with parameters $\ep$ and $d$. Note that by \eqref{eq:dG2} and Lemma \ref{lem:reduced}, we have 
\begin{equation}\label{eq:dGam2}
\delta(\Gamma_R)\geq (\frac{\frac{1}{q}-\frac{1}{q^2}+\gamma}{2+\frac{2}{q}+\gamma}-d-\ep)k\geq \frac{\gamma}{3} k.
\end{equation}
Let $\Gamma_R^1, \dots, \Gamma_R^t$ be the connected components of $\Gamma_R$ and note that by \eqref{eq:dGam2} all of these components have more than $\frac{\gamma}{3} k$ vertices.  For all $i\in [t]$, let $U_i=\bigcup_{V_j\in V(\Gamma_R^i)}V_j$. Without loss of generality suppose $\frac{\gamma}{4}N\leq |U_1|\leq \dots \leq |U_t|.$

\textbf{Case 1} ($|U_t|<(\frac{2(q+1)}{q^2}+\frac{\gamma}{q})n$).  In this case we show that the conditions of Lemma \ref{lem:case1} are satisfied in $G_B[U_1\cup \dots \cup U_t]$ which gives us a blue copy of $W_{2n}$ with center in $U_1$. 

Since $N=(\frac{2(q+1)}{q}+\gamma)n=(\frac{2(q+1)}{q^2}+\frac{\gamma}{q})qn$, then by the case we must have $t\geq q+1$. Consequently, $|U_1|\leq \frac{N}{q+1}=(\frac{2}{q}+\frac{\gamma}{q+1})n$ and thus 
\begin{equation}\label{eq:2tot'}
|U_2|+\dots+|U_t|\geq N-(\frac{2}{q}+\frac{\gamma}{q+1})n = (2+\frac{q\gamma}{q+1})n.
\end{equation}
Furthermore, by \eqref{eq:2tot'} and the fact that $|U_t|<(\frac{2(q+1)}{q^2}+\frac{\gamma}{q})n$, we have
\begin{align*}
|U_2|+\dots+|U_{t-1}| \geq (2+\frac{q\gamma}{q+1})n - |U_t| &> (2+\frac{q\gamma}{q+1})n - (\frac{2(q+1)}{q^2}+\frac{\gamma}{q})n \\
&= (1 + \frac{q^2-2q-2}{q^2} + \gamma(\frac{q}{q+1} - \frac{1}{q}))n\geq  (1 + \frac{\gamma}{2})n.
\end{align*}
Thus, Lemma \ref{lem:case1} is satisfied with $\beta =\frac{\gamma}{2}$, giving us a blue $W_{2n}$.

\textbf{Case 2} ($|U_{t}|\geq (\frac{2(q+1)}{q^2}+\frac{\gamma}{q})n$). Note that since $n > \frac{q^2}{q+1}m$, we have $|U_{t}|\geq (\frac{2(q+1)}{q^2}+\frac{\gamma}{q})n> (2+\frac{q\gamma}{q+1})m$. Let $r$ be the number of parts which have size greater than $(2+\frac{q\gamma}{q+1})m$. If $r \geq q+1$, then $$(2q+2+q\gamma)m\stackrel{\eqref{eq:nqm}}{\geq} N\geq \sum_{i=t-r+1}^t|U_i|>(q+1)(2+\frac{q\gamma}{q+1})m = (2q+2+q\gamma)m,$$ a contradiction. So suppose $1 \leq r \leq q$.

For all $t-r+1\leq i\leq t$, if $\Gamma_R^i$ contains a fractional matching covering at least $\frac{(2+\frac{\gamma}{q})m}{N}k$ vertices of $\Gamma_R^i$, then by Lemma \ref{lem:frac_cycle} we get a red copy of $C_{2m}$. So suppose for all such $i$ there is no such fractional matching, and apply Theorem \ref{thm:frac_GE} to $\Gamma_R^{i}$ to get a partition $\{\mathcal{A}_{i}, \mathcal{C}_{i}, \mathcal{D}_{i}\}$ of $\Gamma_R^{i}$ which gives a resulting partition $\{A_i, C_i, D_i\}$ of $U_i$. By Theorem \ref{thm:frac_GE}, $\mathcal{D}_i$ is an independent set in $\Gamma_R$, and there are no edges between $\mathcal{C}_i$ and $\mathcal{D}_i$ in $\Gamma_R$. Furthermore, there are no edges in $\Gamma_R$ between distinct components $\Gamma_R^i$ and $\Gamma_R^j$. By Lemma \ref{lem:reduced}(i), we have $\delta(\Gamma)\geq (1-\ep)k$ and thus all but an $\ep$-proportion of non-edges in $\Gamma_R$ must be edges in $\Gamma_B$. Since an edge in $\Gamma_B$ implies a blue edge density of at least $1-d$ between the corresponding clusters, the conditions of Lemma \ref{lem:case2} are satisfied in $G_B[U_1\cup \dots \cup U_t]$ with $\beta=\frac{\gamma}{q}$. Note that since $m < \frac{(q+1)n}{q^2}$ and $q \ge 3$, we have $$N-m=(\frac{2(q+1)}{q}+\gamma)n-m> (\frac{2(q+1)}{q}+\gamma)n - \frac{q+1}{q^2}n = (2 + \frac{q-1}{q^2} + \gamma)n.$$ Thus the lower bound on $N$ required for Lemma \ref{lem:case2} is satisfied with $\frac{q-1}{q^2} + \gamma$ in place of $\gamma$.  Applying Lemma \ref{lem:case2} to $G_B[U_1\cup \dots\cup U_t]$ gives a blue $W_{2n}$.
\end{proof}

\section{Odd wheels}\label{sec:odd}

In this section we prove that $R(W_{2n+1})\leq 10n+O(1)$.  We begin with odd counterparts of two lemmas from Section \ref{sec:2up}.

\begin{lemma}[Voss, Zuluaga \cite{VZ}]\label{lem:VZ}
Let $G$ be a graph on $n\geq 2k$ vertices.  If $G$ is 2-connected, non-bipartite, and $\delta(G)\geq k$, then $G$ contains an odd cycle of length at least $2k-1$.  
\end{lemma}

\begin{lemma}[Gould, Haxell, Scott \cite{GHS}]\label{lem:ghs}
Let $d>0$ and let $C := 75 \cdot 10^4/d^5$.  If $G$ is a 2-connected, non-bipartite graph on $N \ge 45C/d^4$ vertices with minimum degree $\delta(G) \ge dN$, then $G$ contains a cycle of length $2t+1$ for all $C \leq 2t+1 \leq \mathrm{oc}(G) - C$, where $\mathrm{oc}(G)$ is the length of the longest odd cycle in $G$.
\end{lemma}

We also use the following special case for the Ramsey numbers of odd cycles.

\begin{lemma}[Faudree, Schelp \cite{FS}; Rosta \cite{Ros}]\label{lem:cycle_ramsey}For all integers $n\geq 2$, 
$R(C_{2n+1}) = 4n + 1$. 
\end{lemma}

Now we prove the main result of this section.  

\begin{proof}[Proof of Theorem \ref{thm:oddW}]
Let $d=\frac{1}{5}$ and $C=75\cdot 10^4/d^5$. Choose constants $c$ and $c'$ so that $c$ is even and $c\gg c'\gg C$.

Let $N=10n+c$ and consider a 2-coloring of $K_N$ with colors 1 (red) and 2 (blue). For all $i\in [2]$, let $G_i$ be the subgraph induced by the edges of color $i$.  For all $i\in [2]$, let $M_i=\{v: d_i(v)\geq 5n+\frac{c}{2}\}$.  Note that since $N=10n+c$ and $c$ is even, we have $M_1\cap M_2=\emptyset$.  Suppose for contradiction that there is no monochromatic $W_{2n+1}$.

We first note that for all $i\in [2]$ and all $v\in M_i$, 
\begin{equation}\label{eq:di}
\delta(G_i[N_i(v)])\geq |N_i(v)|-1-4n\geq \frac{|N_i(v)|}{5}+5\geq n+c'+3.
\end{equation}
where the last two inequalities hold by the choice of $c$ and $c'$. To see the first inequality, suppose there exists $u\in N_i(v)$ such that $|N_{3-i}(u)\cap N_i(v)|\geq 4n+1$.  Then by Lemma \ref{lem:cycle_ramsey}, we have a monochromatic $C_{2n+1}$ in $N_{3-i}(u)\cap N_i(v)$, which gives a monochromatic $W_{2n+1}$.  

\begin{claim}\label{clm:components}
For all $i\in [2]$ and all $v\in M_i$, there exists $X\subseteq N_i(v)$ with $|X|\leq 3$ such that the 2-colored complete graph induced by $N_i(v)\setminus X$ consists of three disjoint cliques of color $i$, each with at least $n+c'$ vertices such that all edges between the cliques have color $3-i$.
\end{claim}

\begin{proofclaim}
Without loss of generality, assume that $i=1$ (red). 

By \eqref{eq:di} and Lemma \ref{lem:2con} (with $k=5$), there exists $X \subseteq N_1(v)$ with $|X| \leq 3$ such that every component of $H:=G_1[N_1(v)]- X$ is 2-connected.  Again, by \eqref{eq:di}, we have $\delta(H)\geq n+c'$.  
Let $V_1, V_2, \dots, V_t$ be the components of $H$, all of which are 2-connected. 

We first establish the following:
\begin{equation}\label{eq:nonbip}
\text{For all $j\in [t]$, $|V_j| \geq 2n + c' \implies H[V_j]$ is bipartite.}
\end{equation}
To see this, suppose for contradiction that there exists $j\in [t]$ such that $H[V_j]$ is non-bipartite and $|V_j| \geq 2n + c'$. Without loss of generality, suppose $j=1$.  Since $\delta(H)\geq n+\frac{c'}{2}$ and $c'\gg C$, we have an odd cycle of length at least $2n+\frac{c'}{2}-1\geq 2n+C+1$ in $H[V_1]$.  Since $2n+c'\geq 45C/d^4$ by the choice of $c'$ (together with \eqref{eq:di} and the fact that $d=\frac{1}{5}$) we may apply Lemma \ref{lem:ghs}, to get a red odd cycle of length exactly $2n+1$ in $V_1\subseteq N_1(v)$.  So we have a red $W_{2n+1}$, a contradiction.

Now we show that $t=3$.  First suppose for contradiction that $t\geq 4$.  Since $\delta(H) \ge n + c'$, we have $|V_j|\geq n + c' + 1$ for all $j\in [t]$. Let $x \in V_1$.   Let $y\in V_2$, $X\subseteq V_3$ and $Y\subseteq V_4$.  Since all edges between the sets are blue, we have that $G_2[\{x,y\}\cup X\cup Y]\simeq K_{1,1,n,n}$ (the complete 4-partite graph with parts of sizes $1,1,n,n$ respectively) which clearly contains $W_{2n+1}$.

Next, suppose $t = 1$. In this case we have $|V_1| \geq 5n + c' > 2n + c'$, which by \eqref{eq:nonbip} implies that $H[V_1]$ is bipartite. So there exists a partition $\{V_1', V_1''\}$ of $V_1$ so that there are no red edges inside $V_1'$ or inside $V_1''$.  We have that, say $|V_1'|\geq \frac{|V_1|}{2}\geq  \frac{5n}{2} + \frac{c'}{2}\geq 2n+2$, and $V_1'$ induces a blue clique which thus contains a blue $W_{2n+1}$, a contradiction.

Finally suppose $t = 2$. Since $|V_1| + |V_2| = |V(H_i)| \ge 5n + c'$, we have without loss of generality $|V_1| \geq \frac{5n}{2} + \frac{c'}{2}$. Because $\frac{5n}{2} + \frac{c'}{2} > 2n + c'$, \eqref{eq:nonbip} implies that $H[V_1]$ is bipartite.  Let $\{V_1', V_1''\}$ be the partition of $V_1$.  Since $\delta(H) \geq n+c'$, we have $|V_1'|, |V_1''|\geq n+c'$. Note that $V_1'$ and $V_1''$ induce blue cliques and all edges from $V_1$ to $V_2$ are blue.  Let $u\in V_1'$.  Choose a path $x_0x_1\dots x_{2n}$ in $G_B[V_1'\setminus \{u\}, V_2]$ with $x_0, x_{2n}\in V_1'\setminus \{u\}$.  Since every edge in $V_1'$ is blue, we have that $x_0x_{2n}$ is blue and thus we have a blue cycle of length $2n+1$ in the blue neighborhood of $u$; that is, a blue $W_{2n+1}$, a contradiction.  

Having established that $t = 3$.  We now show that for all $i\in [3]$,  $V_i$ must be a red clique. Suppose, for contradiction, that say $V_1$ is not a red clique; then there exists at least one blue edge $xy$ with $x,y\in V_1$. We can select $x, y \in V_1$, $X\subseteq V_2$ and $Y\subseteq V_3$ with $|X|=|Y|=n$.  So $G_2[\{x,y\}\cup X\cup Y]\simeq K_{1,1,n,n}$ which contains $W_{2n+1}$.

Thus, $H$ consists of exactly three disjoint red cliques $V_1, V_2, V_3$.  Since these are the components of $H$, we have all blue edges between $V_1, V_2, V_3$. Because $\delta(H) \ge n+c'$, each of these cliques must contain at least $n+c'+1 > n+c'$ vertices, completing the proof of the claim.
\end{proofclaim}

Without loss of generality, suppose $M_1\neq \emptyset$ and let $v\in M_1$.  Apply Claim \ref{clm:components} to $N_1(v)$ to get $X\subseteq N_1(v)$ with $|X|\leq 3$ such that $G_1[N_1(v)]-X$ consists of three disjoint red cliques with vertex sets $V_1, V_2, V_3$ of size $n+c'$ and all blue edges between them. Let $u\in V_1$.  If $u\in M_2$, then by Claim \ref{clm:components} there exists a set $X'\subseteq N_2(u)$ such that $G_2[N_2(u)]-X'$ consists of three disjoint blue cliques of size at least $n+c'$ with all red edges between them (in particular, the red edges induce a graph of chromatic number 3), but this is impossible as $u$ sends blue edges to all of $V_2\cup V_3$ and each of $V_2$ and $V_3$ induces a red clique of size greater than $n$. Thus we must have that $u\in M_1$.

Now applying Claim \ref{clm:components} to $N_1(u)$ we get $X'\subseteq N_1(u)$ with $|X'|\leq 3$ such that $G_1[N_1(u)]-X'$ consists of three disjoint red cliques with vertex sets $V_4, V_5, V_6$ of size at least $n+c'$ and all blue edges between them.  Since $u\in V_1$ and $V_1$ induces a red clique, we have that one of these red cliques, say $V_4$, satisfies  $V_1\setminus (\{u\}\cup X')\subseteq V_4$. Then since $u$ sends blue edges to all of $V_2\cup V_3$, we must have that $V_5$ and $V_6$ are disjoint from $V_2\cup V_3$.  

Let $w\in V_1\setminus (\{u\}\cup X')$ and note that $w$ sends blue edges to all of $V_2\cup V_3\cup V_5\cup V_6$, and $G_2[V_2, V_3]$, $G_2[V_5, V_6]$ are complete blue bipartite graphs with $|V_2|, |V_3|, |V_5|, |V_6|>n$.  If there are no blue edges between $V_2$ and $V_5$, then we have a red clique on $2n+2c'\geq 2n+2$ vertices and thus we have a red $W_{2n+1}$.  So let $x\in V_2$, $y\in V_5$ such that $xy$ is blue.  Likewise if there are no blue edges between $V_2\setminus \{x\}$ and $V_6$, then we have a red clique on at least $2n+2c'-1\geq 2n+2$ vertices and we are again done.  So let $x'\in V_2\setminus \{x\}$ and $y'\in V_6$ such that $x'y'$ is blue.  Since $G_2[V_2, V_3]$ and $G_2[V_5, V_6]$ are complete bipartite graphs with $|V_2|, |V_3|, |V_5|, |V_6|>n$, we have a path from $y$ to $y'$ in $G_B[V_5, V_6]$ of length $2n-3$ and we have a path of length 2 from $x$ to $x'$ in $G_B[V_2, V_3]$ which together with the edges $xy$ and $x'y'$ gives us a blue cycle of length $2n+1$ in the blue neighborhood of $w$; that is, a blue $W_{2n+1}$, a contradiction.  
\end{proof}

\section{Conclusion}\label{sec:Conc}

\subsection{Further improvements and the off-diagonal case}

In order to discuss the value of $R(W_{2m}, W_{2n})$, we may assume by symmetry that $m\leq n$ and write $m=\mu n$.  Then by Theorems \ref{thm:starWheel}, \ref{thm:starWheelExact} for the lower bound, and Theorems \ref{thm:CycleWheel}, \ref{thm:m>n/2}, and \ref{thm:m<n/2} for the upper bound, we have the following for all $q\geq 2$:
\begin{align*}
 (4+\mu+o(1))n &= 
   R(W_{2\mu n}, K_{1,2n})\\
   &\leq R(W_{2\mu n}, W_{2n})\leq 2\cdot R(C_{2\mu n}, W_{2n})=
\begin{cases} 
     4(1+\mu)n+O(1), & \frac{1}{2}\leq \mu< 1\\
     (4q\mu+o(1))n, & \frac{q+1}{q^2}\leq \mu <\frac{1}{q-1}\\
     (4+\frac{4}{q}+o(1))n, & \frac{1}{q}\leq \mu<\frac{q+1}{q^2}\\
\end{cases}.
\end{align*}

It would be interesting to improve the trivial bound $R(W_{2\mu n}, W_{2n})\leq 2\cdot R(C_{2\mu n}, W_{2n})$ (as was done in the case of odd wheels in \cite{ZC} and Theorem \ref{thm:oddW}). In particular, we conjecture the following.  
\begin{conjecture}\label{con:w6}
$R(W_{2n})\leq (6+o(1))n$.
\end{conjecture}

While the methods used in the proof of Theorem \ref{thm:oddW} won't directly help here, we suspect that the methods of \cite{DW} combined with the methods in Section \ref{sec:m<n/2} can be used to approach this problem.  However, we note that even if Conjecture \ref{con:w6} were true, we would still only know that $5n-\frac{1+(-1)^n}{2}\leq R(W_{2n})\leq (6+o(1))n$, and thus some additional work would be required.  

\subsection{Fans and wheels}

From \cite{DW} we have $(3+\sqrt{3})n-8<R(F_n)\leq (5+o(1))n$ and from Theorem \ref{thm:main} we have $5n-\frac{1+(-1)^n}{2}\leq R(W_{2n})\leq (8+o(1))n$. So in particular, we still don't know if $R(F_n)$ and $R(W_{2n})$ are asymptotically equal or not.

A problem which, as far as we know, hasn't been studied directly at all is the following.
\begin{problem}
(Asymptotically) determine the value of $R(F_{n}, W_{2n})$.  It is currently known that
\begin{equation}\label{eq:fanwheel}
5n-\frac{1+(-1)^n}{2}= R(K_{1,2n}, W_{2n})\leq R(F_{n}, W_{2n})\leq 2\cdot R(nK_2, W_{2n}) = 6n.
\end{equation}
\end{problem}

The lower bound in \eqref{eq:fanwheel} is from Theorem \ref{thm:starWheelExact} and the upper bound in \eqref{eq:fanwheel} is (essentially) due to Lin and Li \cite{LL} who proved that $R(nK_2, F_{n})=3n$, but it turns out that their proof implicitly gives $R(nK_2, F_{n})=R(nK_2, W_{2n})=3n$.  We state their proof below with this fact made explicit.  

\begin{theorem}[Lin, Li]\label{lem:LL}
For all positive integers $n$, $R(nK_2, F_{n})= R(nK_2, W_{2n})= 3n$.
\end{theorem}

\begin{proof}The result is clear if $n=1$, so suppose for the rest of the proof that $n\geq 2$.  The lower bound follows by taking disjoint sets $X, Y$ with $|X|=n-1$ and $|Y|=2n$, coloring all edges between the sets red and all edges inside the sets blue. There is no red copy of $nK_2$ and there is no blue copy of $F_{n}$ (and thus no blue copy of $W_{2n}$).

Now suppose we have an arbitrary 2-coloring of $K_{3n}$ and say that a largest red matching $M$ has size $k\leq n-1$.  Let $Z$ be the set of vertices not incident with $M$ and note that $Z$ induces a blue clique (otherwise a red edge in $Z$ would contradict the maximality of $M$). If there exists an edge in $M$ such that each endpoint sends at least 2 red edges to $Z$, then this gives an augmenting path.  So we may suppose that every edge in $M$ has at least one vertex which sends at most one red edge to $Z$. Choose one such vertex from each edge and call this set of $k$ vertices $X$. Since $k < 3n-2k = |Z|$, there exists a vertex $z\in Z$ such that $z$ receives no red edges from $X$, meaning $w$ sends a blue edge to every vertex in $X\cup Z'$. 

Let be the blue graph on $X\cup Z'$; that is, $H=G_B[X\cup Z']$.  We now prove that $H$ contains a cycle on exactly $2n$ vertices which, together with $z$, gives the blue copy of $W_{2n}$ (which in turn contains $F_n$).  Note that $|X|+|Z'|=k+3n-2k-1=3n-k-1\geq 2n$ and for all $v\in X\cup Z'$, we have $$d_B(u,X\cup Z')\geq |Z'|-1=3n-2k-2\geq \frac{3n-k-1}{2} = \frac{|X|+|Z'|}{2},$$ where the second inequality holds since $k\leq n-1$.  So $H$ has at least $2n$ vertices and has minimum degree at least $\frac{|V(H)|}{2}$.  Thus, by Theorem \ref{thm:bon}, $H$ is either pancyclic or a complete balanced bipartite graph; in either case, $H$ contains a copy of $C_{2n}$, which completes the proof.
\end{proof}

\noindent
\tbf{Acknowledgements and AI tool disclosure:}
We thank Yanbo Zhang for alerting us to \cite{ZC}.

This paper was entirely human-written.  Before submitting the paper, we used Gemini to essentially act as a referee.  This process highlighted a number of typos and omitted details which we were then able to address before submission.  The humans who wrote this paper take full responsibility for any remaining errors.  We also used Gemini to help produce Figures \ref{fig:cyclewheel} and \ref{fig:cyclewheelstar}.

\bibliographystyle{abbrv}
\bibliography{references}

\end{document}